\documentclass[10pt]{amsart}

\usepackage{amsmath}
\usepackage{amssymb}
\usepackage[all]{xy}
\usepackage[plainpages]{hyperref}
\usepackage{extarrows}
\usepackage{graphicx,url}
\usepackage{mathrsfs}

\newtheorem{theorem}{Theorem}[section] 
\newtheorem{lemma}[theorem]{Lemma}

\newtheorem{proposition}[theorem]{Proposition}
\newtheorem{corollary}[theorem]{Corollary}

\newtheorem{conjecture}[theorem]{Conjecture}
\numberwithin{equation}{section}

\theoremstyle{definition}
\newtheorem{definition}[theorem]{Definition}
\newtheorem{example}[theorem]{Example}

\newtheorem{remark}[theorem]{Remark}
\newtheorem{ack}{Acknowledgements}

\newcommand{\we}{\wedge}
\renewcommand{\d}{{\rm d}}

\newcommand{\X}{\ensuremath{\mathcal X}}
\renewcommand{\P}{\ensuremath{\mathbb P}}
\renewcommand{\S}{\ensuremath{\mathcal S}}
\newcommand{\A}{\ensuremath{\mathcal A}}
\newcommand{\B}{\ensuremath{\mathcal B}}
\renewcommand{\AA}{\ensuremath{A^\cdot}}

\renewcommand{\S}{\ensuremath{\mathcal S}}
\newcommand{\T}{\ensuremath{\mathcal T}}
\newcommand{\TP}{\ensuremath{\mathbb{TP}}}
\newcommand{\G}{\ensuremath{\mathfrak G}}
\newcommand{\C}{\ensuremath{\mathbb C}}

\newcommand{\R}{\ensuremath{\mathcal {R}}}

\newcommand{\Z}{\ensuremath{\mathbb Z}}

\newcommand{\bbR}{\ensuremath{\mathbb R}}

\newcommand{\la}{\ensuremath{\lambda}}

\renewcommand{\a}{\ensuremath{\mathfrak a}}

\newcommand{\trop}{\operatorname{trop}} 
\newcommand{\Fitt}{\operatorname{Fitt}} 
\newcommand{\Sing}{\operatorname{Sing}}
\newcommand{\crit}{\operatorname{crit}}
\newcommand{\rank}{\operatorname{rank}}
\newcommand{\codim}{\operatorname{codim}}
\newcommand{\im}{\operatorname{im}}
\newcommand{\Jac}{J}
\newcommand{\Ceva}{{\v Ceva\ }}

\newcommand{\xzw}{{\frac{-1}{x+z+w}}}
\newcommand{\xyw}{{\frac{-1}{x+y+w}}}
\newcommand{\xyz}{{\frac{-1}{x+y+z}}}
\newcommand{\yzw}{{\frac{-1}{y+z+w}}}

\renewcommand{\phi}{\varphi}
\renewcommand{\tilde}{\widetilde}
\renewcommand{\bar}{\overline}

\newcommand{\abs}[1]{\left|#1\right|}
\newcommand{\set}[1]{\left\{#1\right\}}

\begin{document}
\title[Vanishing products]{Vanishing products of one-forms and critical points of master functions}

\author[D. Cohen]{Daniel C. Cohen$^1$}
\address{Department of Mathematics, Louisiana State University\\
Baton Rouge, LA 70803, USA}
\email{\href{mailto:cohen@math.lsu.edu}{cohen@math.lsu.edu}}
\urladdr{\href{http://www.math.lsu.edu/~cohen/}
{www.math.lsu.edu/\char'176cohen}}
\thanks{{$^1$}Partially supported 
by National Security Agency grant H98230-05-1-0055}

\author[G. Denham]{Graham Denham$^2$}
\address{Department of Mathematics, University of Western Ontario\\
London, ON  N6A 5B7, Canada}
\urladdr{\href{http://www.math.uwo.ca/~gdenham}%
{www.math.uwo.ca/\char'176gdenham}}
\thanks{{$^2$}Partially supported by a grant from NSERC of Canada}

\author[M. Falk]{Michael Falk}
\address{Department of Mathematics and Statistics, 
Northern Arizona University\\
 Flagstaff, AZ 86011, USA}
\email{\href{mailto:michael.falk@nau.edu}{michael.falk@nau.edu}}
\urladdr{\href{http://www.cefns.nau.edu/~falk/}%
{www.cefns.nau.edu/\char'176falk}}

\author[A. Varchenko]{Alexander Varchenko$^3$}
\address{Department of Mathematics, University of
North Carolina at Chapel Hill\\
Chapel Hill, NC 27599, USA}
\email{\href{mailto:anv@email.unc.edu}{anv@email.unc.edu}}
\urladdr{\href{http://www.math.unc.edu/Faculty/av/}
{www.math.unc.edu/Faculty/av/}}
\thanks{{$^3$}Partially supported by NSF grant DMS-0555327}

\dedicatory{in memory of V. I. Arnol'd}

\keywords{arrangement of hyperplanes, master function, critical locus, resonance variety, Orlik-Solomon algebra, tropicalization}
\subjclass[2010]{Primary
32S22,   
Secondary
55N25,  
52C35,   
14T04
}

\begin{abstract} 
Let \A\ be an 
affine hyperplane arrangement 
in $\C^\ell$ with complement $U$. Let $f_1,\ldots, f_n$ be linear polynomials defining the hyperplanes of \A, and $\AA$ the algebra of differential forms generated by the one-forms $\d \log f_1,\ldots, \d \log f_n$. To each $\la \in \C^n$ we associate the master function $\Phi = \prod_{i=1}^n f_i^{\la_i}$ on $U$ and the closed logarithmic one-form $\omega= \d \log \Phi$.  We assume $\omega$ is a general element of a rational linear subspace $D$ of $A^1$ of dimension $q>1$ such that the map $\bigwedge^k(D) \to A^k$ given by multiplication in $A^\cdot$ is zero for all $p<k\leq q$, and is nonzero for $k=p$.  With this assumption, we prove the critical locus $\crit(\Phi)$ of $\Phi$ has components of codimension at most $p$, and these are intersections of level sets of $p$ rational master functions. We give conditions that guarantee $\crit(\Phi)$ is nonempty and every component has codimension equal to $p$, in terms of syzygies among polynomial master functions. 

If \A\ is $p$-generic,  
then $D$ is contained in the degree $p$ resonance variety $\R^p(\A)$  --  in this sense the present work complements previous work on resonance and critical loci of master functions.  Any arrangement is $1$-generic; we give a precise description of  $\crit(\Phi_\la)$ in case $\la$ lies in an isotropic subspace $D$ of $A^1$, using the multinet structure on \A\ corresponding to $D\subseteq \R^1(\A)$.  This is carried out in detail for the Hessian arrangement. Finally, for arbitrary $p$ and \A, we establish necessary and sufficient conditions for a set of integral one-forms to span such a subspace, in terms of nested sets of \A, using tropical implicitization.
\end{abstract}
\maketitle

\begin{section}{Introduction}
Let $\A=\{H_1,\ldots, H_n\}$ be an arrangement of distinct affine hyperplanes in $\C^\ell$, with complement $U=\C^\ell - \bigcup_{i=1}^nH_i$. Choose a linear polynomial $f_i$ with zero locus $H_i$, for each $i$. Let $\lambda=(\lambda_1\ldots, \lambda_n)\in \C^n$, and consider the {\em master function} 
\[
\Phi_\lambda = \prod_{i=1}^n f_i^{\lambda_i}.
\]
The multi-valued function $\Phi_\lambda$ has a well-defined critical locus 
\[
\crit(\Phi_\la) =\{x\in U \mid \d\Phi_\la(x)=0\}.
\]
Indeed, $\crit(\Phi_\la)$ coincides with the zero locus $V(\omega_\la)$ of the single-valued closed logarithmic one-form
\[
\omega_\la=\d\log(\Phi_\la)=\sum_{i=1}^n \lambda_i \d \log(f_i).
\]
In particular, $\crit(\Phi_\la)$ is unchanged if $\la$ is multiplied by a non-zero scalar.  We are interested in the relation between $\crit(\Phi_\la)$ and algebraic properties of the cohomology class represented by $\omega_\la$ in $H^1(U,\C)$. 

For certain arrangements \A\ and weights \la, the critical points of $\Phi_\la$ yield a complete system of eigenfunctions for the commuting hamilitonians of the $\mathfrak{sl}_n(\C)$-Gaudin model, via the Bethe Ansatz \cite{RV95, ScherbVar03, MukVar05}.
That application was the origin of the term master function, introduced in \cite{Var03}.  Much of that theory depends only on combinatorial properties of arrangements, and can be formulated in that general setting -- see \cite{Var10}.

Let $\AA$ denote the graded \C-algebra of holomorphic differential forms on $U$ generated by  $\{\d \log(f_i) \mid 1\leq i \leq n\}$. By a well-known result of Brieskorn, the inclusion of $\AA$ into the de Rham complex of $U$ induces an isomorphism in cohomology, and thus $\AA\cong H^\cdot(U,\C)$, see \cite{A69,Bries73}. In particular $A^1 \cong \C^n$. Since $\omega_\la\we \omega_\la =0$, left-multiplication by $\omega_\la$ makes $\AA$ into a cochain complex $(\AA,\omega_\la)$.  For generic $\la$, $H^p(\AA,\omega_\la)=0$ for $p<\ell$, and $\dim H^\ell(\AA,\omega_\la)=|\chi(U)|$, see \cite{SV91,Yuz95}.  At the same time, for generic $\la$, $\Phi_\lambda$ has $|\chi(U)|$ isolated, nondegenerate critical points \cite{Var95, OT95,Silv96}. Here we are concerned with $H^p(\AA,\omega_\la)$ and $\crit(\Phi_\la)$ when $\la$ is not generic.

For $p<\ell$, the $\omega$ for which $H^p(\AA,\omega)\neq 0$ comprise the $p^{\rm th}$ {\em resonance variety} $\R^p(\A)$ of \A, a well-studied invariant of $\AA$.  On the other hand, precise conditions on $\la$ guaranteeing that $\Phi_\la$ has $|\chi(U)|$ isolated critical points are not known. There are also examples where the critical points of $\Phi_\la$ are isolated but degenerate.

In some cases $\crit(\Phi_\la)$ is positive-dimensional. If \A\ is a discriminantal arrangement, in the sense of \cite{SV91}, then for certain 
 choices of integral weights $\la$ arising from a simple Lie algebra 
 $\mathfrak g$, $\crit(\Phi_\la)$ has components of the same positive
 dimension     \cite{ScherbVar03,MukVar04,MukVar05b}. In the particular case $\mathfrak g=\mathfrak{sl}_2(\C)$
 of this situation, 
the codimension of $\crit(\Phi_\la)$ is $\ell-1$.  In this case,  
 it was shown in \cite{CV} that $\omega_\la \in \R^{\ell-1}(\A)$ for these $\la$, with the rank of the skew-symmetric part of $H^{\ell-1}(\AA,\omega_\la)$ equal to 
the number of components of   $\crit(\Phi_\la)$.
%

Our work in  \cite{CDFV08} provides a weak generalization of these results. There we study the {\em universal critical set}, the set $\Sigma$ of pairs $(x,a)$ such that $x \in V(\omega_a)$. For fixed $\lambda$, $\crit(\Phi_\la)$ is the $a=\la$ slice of $\Sigma$. Let $\bar{\Sigma}$ be the Zariski closure of $\Sigma$ in $\C^\ell \times \C^n$, and $\bar{\Sigma}_\la$ the $a=\la$ slice of $\bar{\Sigma}$. In \cite{CDFV08} we show, if $\omega_\la \in \R^p(\A)$, then $\bar{\Sigma}_\la$ has codimension at most $p$, provided \A\ is tame and either $p\leq 2$ or \A\ is free. See \cite{CDFV08} for definitions of free and tame arrangements; any affine arrangement in $\C^2$ is tame. It is not true in general that $\bar{\Sigma}_\la$ is the closure of $\Sigma_\la$. Indeed, $\Sigma_\la$ may be empty under the given hypotheses - that is, $\bar{\Sigma}_\la \subseteq \C^\ell \times \C^n$ may lie over $\bigcup_{i=1}^n H_i$.

In this paper, we obtain somewhat more precise information on $\crit(\Phi_\la)$, for more general arrangements, but impose a different hypothesis on $\omega_\la$. Namely, we assume that $\omega_\la$ has a {\em decomposable cocycle}, that is,  there exists $\psi\in A^p$ such that $\omega_\la \we \psi=0$, and $\psi$ is a product of $p$ elements of $A^1$, whose linear span does not include $\omega_\la$.

We say a subspace $D$ of $A^1$ is {\em singular} if the multiplication map $\bigwedge^q(D) \to A^q$ is zero, where $q=\dim D$. Let $p$ be maximal such that $\bigwedge^p(D) \to A^p$ is not the zero map. If $\omega = \d \log(\Phi) \in D-\{0\}$, then $\omega$ can be included in a basis of $D$, and any $p$-fold product of the other basis elements is a decomposable $p$-cocycle for $\omega$.

We assume that $D$ is a rational subspace of $A^1$, that is, $D$ has a basis $\Lambda=\{\omega_{\xi_1},\ldots,\omega_{\xi_q}\}$ with each $\omega_{\xi_j}$ an integer linear combination of $\d\log(f_1), \ldots, \d\log(f_n)$.  Associated to $\Lambda$ is a rational mapping $\Phi_\Lambda =(\Phi_{\xi_1}, \ldots, \Phi_{\xi_q}) \colon \C^\ell \rightarrowtail\C^q$, whose image is a quasi-affine subvariety $Y=Y_\Lambda$ of $\C^q$. The dimension of $Y$ is $p$.  If $\omega=\d \log(\Phi) \in D$, then the critical locus $\crit(\Phi)$ is consists of fibers of $\Phi_\Lambda$ and singular points of $\Phi_\Lambda$. In particular, for generic $\omega \in D$, $\crit(\Phi)$ has codimension at most $\dim(Y)=p$.  We obtain more precise conclusions in case the projective closure $\bar{Y}$ is a curve ($p=1$) or a hypersurface ($p=q-1$), or $\bar{Y}$ is nonsingular and meets the coordinate hyperplanes transversely. If $\bar{Y}$ is linear, of any codimension, with $Y=\bar{Y} \cap (\C^*)^q$ and $\Phi_\Lambda$ nonsingular, we get a complete description of the $\crit(\Phi)$ for $\d \log(\Phi)\in D$, in terms of critical loci of master functions on the complement of the rank-$p$ arrangement cut out on $\bar{Y}$ by the coordinate hyperplanes.

Every component of $\R^1(\A)$ is a rationally-defined and isotropic linear subspace \cite{LY00}, and every element of $\R^1(\A)$ has a decomposable cocycle. Moreover, by the theory of multinets and \Ceva pencils \cite{FY07}, we can choose $\Lambda$ so that the variety $Y_\Lambda$ corresponding to a component of $\R^1(\A)$ is linear. We carry out the entire analysis in detail in this case, with special attention to the Hessian arrangement, the one case we know for which $Y_\Lambda$ is not a hypersurface.


Our approach lends itself to tropicalization, using the main result of \cite{DFS07}. Using the nested set subdivison of the Bergman fan \cite{FeStu05}, we derive a rank condition for a product of integral one-forms $\omega_{\xi_1} \we \cdots \we \omega_{\xi_q}$ to vanish.
The rank condition can be used in case \A\ is $p$-generic to give a combinatorial description of the $(p+1)$-tuples of integral forms in $A^1$ whose product vanishes, analogous to the description of $\R^1(\A)$ in terms of neighborly partitions -- see \cite{BFW11}.

The outline of this paper is as follows. In  Section 2 we introduce Orlik-Solomon algebras and resonance varieties, prove a general result about zero loci of differential forms, and compute critical loci directly for some examples, including the Hessian arrangement. In Section 3 we consider logarithmic one-forms with decomposable $p$-cocycles satisfying the rationality criterion above, obtaining a precise description of their zero loci, especially in case $\Phi_\Lambda$ is nonsingular and $\bar{Y}_\Lambda$ is a hypersurface meeting the coordinate hyperplanes transversely. We revisit the examples from Section 2. In Section 4 we treat the case where $\bar{Y}_\Lambda$ is linear, returning to the example of the Hessian arrangement. In Section 5 we formulate a test for existence of decomposable cocycles using tropical implicitization.
\end{section}

\begin{section}{Resonance, vanishing products, and zeros of one-forms}

It will be more convenient for us to consider arrangements of projective hyperplanes in complex projective space $\P^\ell$. Let $[x_0: \cdots : x_\ell]$ be homogeneous coordinates on $\P^\ell$, and let $\alpha_i \colon \C^{\ell+1} \to \C$ be a nonzero homogeneous linear form, for $0\leq i\leq n$. Assume without loss that $\alpha_0(x)=x_0$. Let $H_i=\ker(\alpha_i)$, considered as a projective hyperplane in $\P^\ell$, and let $\A=\{H_0,\ldots, H_n\}$. We will denote the corresponding linear hyperplanes in $\C^{\ell+1}$ by $cH_i$, comprising the central arrangement $c\A=\{cH_0, \ldots, cH_n\}$. Let $U=\P^\ell -\bigcup_{i=0}^n H_i$. We identify $[1:x_1:\cdots :x_\ell] \in \P^\ell - H_0$ with $(x_1,\ldots, x_\ell)\in \C^\ell$,  and set
\[
f_i(x_1,\ldots, x_\ell)=\alpha_i(1,x_1,\ldots, x_\ell)
\]
for $1\leq i\leq n$. Then we recover the affine arrangement \A\ of the Introduction, with the same complement $U$. In $\P^\ell$, $U$ is the complement of the singular projective hypersurface defined by 
\[
Q=\prod_{i=0}^n \alpha_i,
\]
the (homogeneous) {\em defining polynomial} of \A. 

\subsection{The projective Orlik-Solomon algebra}
Let $\Omega^\cdot(U)$ be the complex of holomorphic differential forms on $U$. The {\em Orlik-Solomon algebra} of \A\ is the subalgebra $\AA(\A)$ of $\Omega^\cdot(U)$ generated by $\d \log(f_i),  \ 1\leq i\leq n$, as in the Introduction.

We will also study $\AA(\A)$ in homogeneous coordinates. Let $\omega_i= \d \log(\alpha_i)$ for $0\leq i\leq n$, and let $\AA(c\A)$ be the algebra of holomorphic forms on $\C^{\ell+1}$ generated by $\omega_0,\ldots, \omega_n$. Define $\partial \colon \AA(c\A) \to \AA(c\A)$ by 
\[
\partial(\omega_{i_1}\we \cdots \we \omega_{i_k}) = \sum_{j=1}^k (-1)^{j-1} \omega_{i_1}\we \cdots \we \widehat{\omega}_{i_j} \we \cdots \we \omega_{i_k}
\]
and extending linearly. Then $\partial$ is a graded derivation of degree -1, and  
\[
\partial \bigl(\sum_{i=0}^n c_i \omega_i\bigr)=\sum_{i=0}^n c_i.
\]  

In general, a holomorphic $p$-form on $\C^{\ell+1}-\{0\}$ descends to a well-defined form on $\P^\ell$ if and only if it is $\C^*$-invariant and its contraction along the Euler vector field $\sum_{i=0}^n x_i \frac{\partial}{\partial x_i}$ vanishes, see \cite{Dimca92}.  This contraction, on $\AA(c\A)$, is given by $\partial$, and $\AA(c\A)$ consists of $\C^*$-invariant forms. Then we may identify $\AA(\A)$ with the subalgebra $\ker(\partial)$ of $\AA(c\A)$. This is easily seen to coincide with the subalgebra of $\AA(c\A)$ generated by $\ker(\partial)\cap A^1(c\A)$.

With our choice of coordinates, $\d \log(f_i) = \omega_i-\omega_0$ under this identification. $\{\omega_1-\omega_0, \ldots, \omega_n - \omega_0\}$ generates \AA\ by the remark above. Also, $(\AA(c\A),\partial)$ is an exact complex, so that $\im(\partial)=\ker(\partial)=\AA(\A)$,  see \cite{OT92}.

There is a well-known presentation of $\AA(c\A)$ as a quotient of the exterior algebra $E^\cdot=\bigwedge^\cdot(e_0, \ldots, e_n)$. For $C=\{i_1,\ldots, i_k\} \subseteq \{0, \ldots,n\}$, write $e_C=e_{i_1} \cdots  e_{i_k}\in E^k$. Say $C$ is a {\em circuit} of $c\A$ if $C$ is minimal with the property that 
\[
\codim \bigcap_{j\in C} H_j < |C|.
\]
 Then $\AA(c\A)$ is isomorphic to $E^\cdot/I$, where 
 \[
 I=\bigl(\partial e_C \mid C \ \text{is a circuit of} \ c\A \bigr).
 \]

\subsection{Resonance varieties} Let $\omega = \sum_{i=0}^n \la_i \omega_i$, where $\la=(\la_0,\ldots, \la_n) \in \C^{n+1}$. Assume that $\partial \omega =\sum_{i=0}^m \la_i =0$. Since $\d \log(f_i) = \omega_i-\omega_0$, $\omega=\sum_{i=1}^n \lambda_i\, \d \log(f_i)$. 

Then $\omega \in A^1$, and $\omega \we \omega=0$, so we obtain a cochain complex

\[
0 \longrightarrow A^0  \xlongrightarrow{\omega \we -} \cdots \xlongrightarrow{\omega \we -} A^p \xlongrightarrow{\omega \we -} \cdots \xlongrightarrow{\omega \we -} A^\ell \longrightarrow 0. 
\]

Let 
\begin{align*}
Z^p(\omega)&=\{\psi \in A^p \mid \omega \we \psi=0\},\\
B^p(\omega)&=\{\psi \in A^p \mid \psi=\omega \we \phi \ \text{for some} \ \phi \in A^{p-1}\}, \ \text{and} \\
H^p(\AA,\omega)&=Z^p(\omega)/B^p(\omega).
\end{align*}
Then
\[
\R^p(\A)=\{ \omega \in A^1 \mid H^p(\AA,\omega)\neq 0\}
\]
is, by definition, the $p^{\rm th}$ {\em resonance variety} of \A.

As observed in the Introduction, 
\[
\omega=\d \log(\Phi)=\frac{\d \Phi}{\Phi},
\]
where $\Phi=\prod_{j=1}^n f_j^{\la_j}$, and $\crit(\Phi)$ coincides with the zero locus of $\omega$. In homogeneous coordinates, $\Phi$ is given by $\prod_{j=0}^n \alpha_j^{\la_j}$. 

\subsection{Zeros of forms} We start with an elementary observation about products and zeros of differential forms. If $\psi\in \Omega^k(U)$, for 
some $k$, $0 \leq k \leq \ell$, let 
\[
V(\psi)=\{x \in U \mid \psi(x)=0\},
\]
a quasi-affine subvariety of $\C^\ell$. Let $U(\psi)=U - V(\psi)$.

\begin{proposition} Suppose $\omega\in \Omega^1(U)$ and $\psi \in \Omega^p(U)$ satisfy $\omega \we \psi = 0$. Then every component of $V(\omega) - V(\psi)$ has codimension less than or equal to $p$. 
\label{prop:zeros}
\end{proposition}

\begin{proof} We may write $\omega=\sum_{i=1}^\ell b_i\d x_i$ for
some holomorphic functions $b_1, \ldots b_\ell$ on $U$. Then 
\[
V(\omega)=\bigcap_{i=1}^\ell V(b_i).
\]
Similarly, $\psi=\sum_{I}A_I\d x_I$ for some holomorphic functions 
$A_I$, where $I$ ranges over all subsets $I=\{i_1,\ldots, i_p\}_<$ of
$\{1,2,\ldots,\ell\}$, and $\d x_I=\d x_{i_1}\wedge\cdots\wedge \d x_{i_p}$.  (The subscript ``$<$'' is meant to indicate that the elements of $I$ are listed in increasing order.) Set $U_I=U(A_I)$, and let $S_I$ denote the coordinate ring of $U_I$, i.e., $S_I$ is $\C[x_1,\ldots, x_\ell]$, localized at $A_I$. Then 
\[
U(\psi)=\bigcup_I \, U_I.
\]

The equation $\omega\we \psi=0$ says, for each subset  $J$ of $\{1,\ldots,\ell\}$ of size $p+1$,
\begin{equation}
\sum_{i\in J}\sigma(i,J) b_iA_{J-\{i\}}=0.
\label{eqn:one}
\end{equation}
Here $\sigma(i,J)=\pm 1$ depending on the position of $i$ in $J$. 

We have 
\[
V(\omega)\cap U(\psi)=\bigcup_I V(\omega) \cap U_I.
\]

Fix $I=\{i_1,\cdots, i_p\}_<$. For each $i\not \in I$, set $J=I\cup\{i\}$ in equation \eqref{eqn:one}. Since $A_I\neq 0$ on $U_I$, one can solve for $b_i$ in terms of $b_{i_1}, \ldots, b_{i_p}$. This means $b_i$ lies in the ideal $(b_{i_1},\ldots, b_{i_p})$ of $S_I$. Since this holds for every $i\not \in I$, the defining ideal of $V(\omega) \cap U_I$  in $S_I$ is contained in $(b_{i_1},\ldots, b_{i_p})$. Then each irreducible component of $V(\omega) \cap U_I$ has codimension less than or equal to the codimension of $(b_{i_1},\ldots, b_{i_p})$, which is at most $p$.  Since the $U_I$ cover  $U(\psi)$, the result follows.
\end{proof}

\begin{corollary} If 
\[
\bigcap \{V(\psi) \mid \psi \in \Omega^p(U), \omega \we \psi =0\} =\emptyset,
\]
then every component $V(\omega)$ has codimension less than or equal to $p$. 
\label{cor:zeros}
\end{corollary}

\begin{corollary} Suppose $X$ is a component of $V(\omega)$ of codimension $c$. 
If $\psi$ is a $p$-form satisfying $\omega\we\psi=0$ and $p<c$, then $X\subseteq V(\psi)$.
\label{cor:badcomp}
\end{corollary}

\begin{remark} The preceding results go through without change for any smooth complex analytic variety $U$, interpreting $x_1,\ldots, x_\ell$ as local holomorphic coordinates on $U$.
\label{rmk:zeros}
\end{remark}

A $p$-form $\psi$ satisfying $\omega \we \psi=0$ will be called a $p$-{\em cocycle} for $\omega$. We say $\psi$ is {\em trivial} if $\psi=\omega \we \phi$ for some $\phi \in \Omega^{p-1}(U)$. If $\psi$ is a trivial cocycle for $\omega$, then $V(\omega) \subseteq V(\psi)$. 

The trivial cocycle condition $\psi=\omega \we \phi$ is generally difficult to characterize. We propose the following conjecture, the converse to the observation above.

\begin{conjecture}
If $\psi\in A^p$, then $\psi=\omega\we \phi$ for some $\phi \in A^{p-1}$ if and only if $\omega\we \psi=0$ and $V(\omega)\subseteq V(\psi)$. 
\end{conjecture}
The conjecture is not hard to prove directly in case $p=1$, and the statement for any $p$ follows from results of \cite{Var10} if $\omega \in A^1$ is generic. 






%

\subsection{Examples} Our first example is linearly equivalent to the rank-three braid arrangement.

\begin{example}
\label{ex:braid}
Let $\A=\{H_0,\ldots, H_5\}$ be the arrangement with defining polynomial
\[
Q=xyz(x-y)(x-z)(y-z),
\]
with the hyperplanes labelled according to the order of factors in $Q$.

For $a,b,c \in \C$, not all zero, with $a+b+c=0$, let 
\[
\Phi_{abc}=[x(y-z)]^a[y(x-z)]^b[z(x-y)]^c.
\]
Then  $\omega_{abc}:=\d \log(\Phi_{abc})=a(\omega_0+\omega_5) + b(\omega_1+\omega_4) + c(\omega_2+\omega_3)$.

One computes  
\[
\begin{split}
\omega_{abc}&= \left[\d x\,\d y\,\d z\right]\left[\begin{array}{c}b_1\\b_2\\b_3\end{array}\right] \\
& = \left[\d x\,\d y\,\d z\right]\left[\begin{array}{ccc}
1/x & 1/(x-z) & 1/(x-y)\\
1/(y-z) & 1/y & 1/(y-x)\\
1/(z-y) & 1/(z-x) & 1/z\\
\end{array}\right]\left[\begin{array}{c}a\\b\\c\end{array}\right].
\end{split}
\]
The zero locus $V(\omega_{abc})$ is given by the vanishing of $b_1, b_2$, and $b_3$. The kernel of the matrix is spanned by $(x(y-z),y(z-x),z(x-y))$, so
$[x:y:z]\in V(\omega_{abc})$ if and only if $[x(y-z):y(z-x):z(x-y)]=[a:b:c]$. Since $a+b+c=0$, this is equivalent to 
$[x(y-z):y(z-x)]=[a:b]$, i.e., 
\[
\frac{x(y-z)}{y(z-x)}=\frac{a}{b}
\]
or, more symmetrically, 
\[
\frac{a}{x}+\frac{b}{y}+\frac{c}{z}=0.
\]
If any of $a, b$, or $c$ are zero, then $\crit(\Phi_{abc})=V(\omega_{abc})$ is empty. Otherwise $\crit(\Phi_{abc})$  has codimension $1$. Moreover, $\crit(\Phi_\la)$ is a level set of the master function $\frac{x(y-z)}{y(z-x)}$.

Let $D=\{\omega_{abc} \mid a+b+c=0\}$.
Then $D \subseteq A^1$ and, for any $\omega \in D$, $Z^1(\omega)=D$.
Thus $\psi\in Z^1(\omega_{abc})$ if and only if $\psi=  \d\log\Phi_{a'b'c'}$ with $a'+b'+c'=0$.
Then $\psi$ has zero locus
given by 
\[
\frac{a'}{x}+\frac{b'}{y}+\frac{c'}{z}=0,
\]
which one can see is disjoint from  
$V(\omega_{abc})$ if and only if $\psi\not \in B^1(\omega_{abc})$. The Zariski closure of every nonempty critical set contains the four points $[0:0:1]$, $[1:1:1]$, $[1:0:0]$, and $[0:1:0]$ - see Figure~\ref{fig:braidpic}.

\begin{figure} 
\includegraphics[width=6cm]{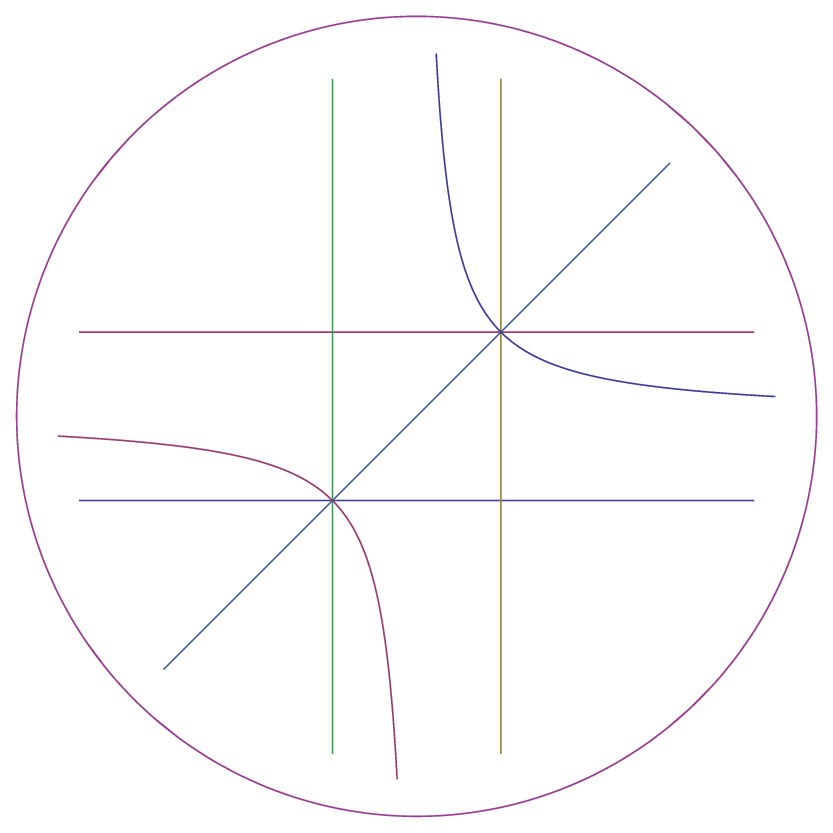}
\caption{The braid arrangement with $\crit(\Phi_{1,1,-2})$.}

\label{fig:braidpic}
\end{figure}

\end{example}

Here is a rank-four example.
\begin{example} \label{ex:nondecomp}
Let $\A = \{H_0,\ldots, H_7\}$ be the arrangement of eight planes in $\P^3$ with defining polynomial 
\[
Q=xyzw(x+y+z)(x+y+w)(x+z+w)(y+z+w).
\]
The dual point configuration consists of the four vertices and four face-centers of the 3-simplex. Fix $a,b,c,d \in \C$ and let 
\[
\Phi=\bigg({x\over y+z+w}\bigg)^a\bigg({y\over x+z+w}\bigg)^b
\bigg({z\over x+y+w}\bigg)^c\bigg({w\over x+y+z}\bigg)^d.
\]

The one-form $\omega=\d\log\Phi$ belongs to a $4$-dimensional component of $\R^2(\A)$. If none of $a,b,c,d$ are zero, then $H^1(A,\omega)=0$ and $H^2(A,\omega)\cong \C$.
A nontrivial 2-cocycle for $\omega$ is given by
\[
\psi = b\cdot \partial(\omega_{167}) + c\cdot\partial(\omega_{257}) +
d\cdot\partial(\omega_{347}).
\]

One sees that $\omega$ is equal to the product
\[
\left[\d x\,\d y\,\d z\,\d w\right]\left[\begin{array}{cccc}
\frac{1}x & \xzw & \xyw & \xyz\\
\yzw & \frac{1}y & \xyw & \xyz\\
\yzw & \xzw & \frac{1}z & \xyz\\
\yzw & \xzw & \xyw & \frac{1}w\\
\end{array}\right]\left[\begin{array}{c}a\\b\\c\\d\end{array}\right].
\]
Computing the kernel of the matrix, we see that $[x:y:z:w]\in V(\omega)$ if and only if the vector $(x(y+z+w),y(x+z+w),z(x+y+w),w(x+y+z))$
is proportional to $(a,b,c,d)$, giving the equations
\[
\frac{x(y+z+w)}{w(x+y+z)}=\frac{a}{d},\quad
\frac{y(x+z+w)}{w(x+y+z)}=\frac{b}{d},\quad
\frac{z(x+y+w)}{w(x+y+z)}=\frac{c}{d}.
\]
$V(\omega)$ has a component of codimension two, given by 
\[
x+y+z+w=0, \quad \frac{a}{x}+\frac{b}{y}+\frac{c}{z}+\frac{d}{w}=0.
\]
For general $a,b,c,d$, the remaining components of $V(\omega)$ consist of four additional points in $\P^3$.
It follows from Corollary~\ref{cor:badcomp} that the cocycle $\psi$ vanishes 
on the isolated points of $V(\omega)$.  This can also be verified here by direct
computation.
\end{example}

\begin{example}
\label{ex:hessian}
The Hessian arrangement consists of the $12$ lines through the inflection points of a 
nonsingular cubic in $\P^2$. It is the only known arrangement of rank greater than two 
that supports a global component 
of $\R^1(\A)$ of dimension greater than two. 
That is, there is an element $\omega \in A^1$ which has poles along every hyperplane of \A, and satisfies $\dim H^1(\AA,\omega)> 1$.

Any nonsingular cubic is equivalent to 
\begin{equation}
x^3+y^3+z^3 -3txyz =0
\label{eq:cubics}
\end{equation}
up to projective transformation, for some $t\in \C$. These cubics have the same inflection points. Then, up to projective transformation, the Hessian arrangement \A\ is defined by
\begin{multline*}
Q=xyz(x+y+z)(x+y+\zeta z)(x+y+\zeta^2 z)(x+\zeta y+z)(x+\zeta y+\zeta z)
(x+\zeta y+\zeta^2 z)\\
\cdot (x+\zeta^2 y+z)(x+\zeta^2 y+\zeta z)(x+\zeta^2y+\zeta^2 z),
\end{multline*}
where $\zeta=e^\frac{2\pi i}{3}$.   
The 12 lines of \A\ are the irreducible components of the four singular cubics in the family \ref{eq:cubics}, corresponding to $t=\infty, 1, \zeta$, and $\zeta^2$. See \cite[Example 6.30]{OT92}.

Numbering the hyperplanes in order, these singular cubics are given by
\begin{alignat*}{3}
P_0 &=\alpha_0\alpha_1\alpha_{2\phantom{2}}\ &= &\ xyz, \\
P_1 &=\alpha_3\alpha_8\alpha_{10}\ & = &\ x^3+y^3+z^3-3xyz,\\
P_2 &=\alpha_4\alpha_6\alpha_{11}\ &= &\ x^3+y^3+z^3-3\zeta xyz,  \ \text{and} \\
P_3 &=\alpha_5\alpha_7\alpha_{9\phantom{9}}\ &= &\ x^3+y^3+z^3-3\zeta^2 xyz.
\end{alignat*}



For $(a_1,a_2,a_3)\in \C^3$ let $\omega=\omega_{a_1a_2a_3}=
\d \log(\Phi)$, where
\begin{equation*}\label{eq:one}
\Phi=\Phi_{a_1a_2a_3}=\left(\frac{P_1}{P_0}\right)^{a_1}\left(\frac{P_2}{P_0}\right)^{a_2}\left(\frac{P_3}{P_0}\right)^{a_3}.
\end{equation*}
Let $D\subseteq A^1$ be the space of all such forms. Then $H^1(\AA, \omega)\cong D/\C\omega$ has dimension two.  

%
 
As in the previous examples, we write
\[
\omega=\begin{bmatrix}\d x,\d y,\d z \end{bmatrix}
M \begin{bmatrix}a_1\\a_2\\a_3\end{bmatrix},
\]
where $M$ is a $3\times3$ matrix of rational functions, the Jacobian of 
\[(\log(P_1/P_0), \log(P_2/P_0), \log(P_3/P_0)).\] 
Then $\omega(x)=0$ for $x\in U$ if and only if $a=(a_1,a_2,a_3)$ lies in the kernel of $M(x)$.

The matrix $M$ has rank $1$.  In fact one finds that
$M=vw^T$ where
\[
\begin{array}{ccc}
v=\begin{bmatrix}\frac{2x^3-y^3-z^3}{x} \\ \frac{x^3-2y^3+z^3}{y} \\ \frac{x^3+y^3-2z^3}{z} \end{bmatrix} & \text{and} & 
w=\begin{bmatrix}\frac{1}{P_1} \\ \frac{1}{P_2} \\ \frac{1}{P_3} \end{bmatrix}.
\end{array}
\]

The critical equation becomes $vw^Ta=0$, which is satisfied if and only if $w^Ta=0$ or all components of $v$ vanish. The latter occurs at the points given by $x^3=y^3=z^3$, which are the common inflection points of the cubics (\ref{eq:cubics}). In particular those points do not lie in the complement $U$. Thus $\crit(\Phi)$ is defined by the single equation

\begin{equation}
\frac{a_1}{P_1}+\frac{a_2}{P_2}+ \frac{a_3}{P_3}=0.
\label{eqn:hesscrit}
\end{equation}

Then $\crit(\Phi)$ is empty or has codimension one and degree six. Working in the torus $xyz\neq 0$, set
\[
T=\frac{x^3+y^3+z^3}{xyz}.
\]
Then equation \eqref{eqn:hesscrit} is equivalent to 
\begin{equation}
\frac{a_1}{T-3}+\frac{a_2}{T-3\zeta}+\frac{a_3}{T-3\zeta^2}=0,
\label{eqn:quad}
\end{equation}
which becomes a quadratic $AT^2+BT+C=0$ in $T$. 

Then, if $a=(a_1,a_2,a_3)$ is generic, $\crit(\Phi)$ has two irreducible components. Each component is the intersection with $U$ of a level set $T=3t$ of $T$. These are nonsingular fibers in the pencil (\ref{eq:cubics}), meeting each other and \A\ at their nine common inflection points. So $\crit(\Phi)\subseteq U$ has two connected components. Every pair of nonsingular fibers appears as $\crit(\Phi)$ for some $a$. The discriminant $B^2-4AC$ defines a hypersurface in $(a_1,a_2,a_3)$-space for which the corresponding critical locus $\crit(\Phi)$ has a single nonreduced component, and every nonsingular fiber can appear. 

 For some values of $a$, $\crit(\Phi)$ is empty or has only one reduced component. 
 This is easiest to see by clearing fractions in \eqref{eqn:hesscrit}, to obtain 
\begin{equation}
a_1P_2P_3+a_2P_1P_3+a_3P_1P_2=0.
\label{eqn:hesscrit2}
\end{equation}
When one component of the variety defined by \eqref{eqn:hesscrit2} is $P_i=0$ or $P_0=xyz=0$, then $\crit(\Phi)$ has one reduced component. This occurs if $a$ is a generic point on $a_i=0$ or $a_1+a_2+a_3=0$. If $a_i=0$ for some $i$ and $a_1+a_2+a_3=0$, or if $a_i\neq 0$ for only one $i$, then \eqref{eqn:hesscrit2} becomes $P_iP_j=0$, and $\crit(\Phi)=\emptyset$. If $a$ is a cyclic permutation of $(0,\zeta,1)$ or equals $(1,\zeta, \zeta^2)$, up to scalar multiple, then \eqref{eqn:hesscrit2} becomes $P_i^2=0$, and $\crit(\Phi)=\emptyset$.

Each cocycle of $\omega=\omega_{a_1a_2a_3}$ has the form 
$\psi=\omega_{b_1b_2b_3}$ for some $b_1,b_2,b_3$. So we see that $V(\psi)$ and $V(\omega)$ can have a component in common, but if $\omega$ and $\psi$ are not proportional, then $V(\omega)-V(\psi)$ is nonempty and has codimension one.

\end{example}
\end{section}

\begin{section}{Decomposable cocycles}
We will now assume $\psi$ is a cocycle for $\omega \in A^1$ and $\psi$ is a product of logarithmic one-forms. Then $\omega$ is a factor in a vanishing product of $p+1$ one-forms in \AA. To carry out our geometric analysis it is necessary to work initially over the integers, although eventually the results extend to \C-linear combinations of the original integral weights. For that reason we state the result in terms of subspaces of $A^1$.

We will be dealing with rational functions parametrizing affine or projective algebraic varieties. For that reason we formulate and prove our results algebraically. For any quasi-projective variety $X$, write $\C[X]$ for the ring of regular functions on $X$, and $\C(X)$ for the field of rational functions on $X$. If $X$ is a subvariety of projective space, then elements of $\C[X]$  are represented by homogeneous polynomials and elements of $\C(X)$ by homogeneous rational functions of degree zero. 
If $\varphi \colon R \to S$ is a homomorphism of \C-algebras, denote by $\Omega_{S|R}$ the $S$-module of Kahler differentials of $S$ over $R$. Write $\Omega[X]=\Omega_{\C[X]|\C}$ and $\Omega(X)= \Omega_{\C(X)|\C}$. 
Elements of $\Omega[X]$ (resp. $\Omega(X)$) are polynomial (resp. rational) one-forms on $X$.

\subsection{Singular subspaces of $A^1$} Let $D$ be a subspace of $A^1$. We call $D$ a {\em singular subspace} if the multiplication map $\bigwedge^{q}(D) \to A^q$ is the zero map, where $q=\dim D$. The {\em rank} of $D$ is the largest $p$ such that $\bigwedge^p(D) \to A^p$ is not trivial.
We say $D$ is {\em rational} if $D$ has a basis $\{\omega_{\xi_1},\ldots,\omega_{\xi_q}\}$, with $\xi_i=(\xi_{i0},\ldots, \xi_{in})\in \Z^{n+1}$ for $1\leq i\leq q$.
Then $\Phi_{\xi_i}=\prod_{j=1}^n f_j^{\xi_{ij}}=\prod_{j=0}^n \alpha_j^{\xi_{ij}}$ is a single-valued rational function on $\P^\ell$, regular on $U$. (Recall $\sum_{j=0}^n\xi_{ij}=0$.)

We apply the following general result. It is an easy consequence of the implicit function theorem, but we give an algebraic proof that holds over any algebraically-closed field of characteristic zero. See \cite{Har77} and \cite{Eis95} for the relevant background on K\"ahler differentials.

\begin{proposition} Suppose $F_1,\ldots, F_q$ are rational functions on $\C^\ell$, and 
\[
F =(F_1,\ldots, F_q) \colon \C^\ell \rightarrowtail \C^q.
\]
Then the image of $F$ has dimension less than $k$ if and only if 
\[
\d F_{i_1}\we \cdots \we \d F_{i_k}=0
\] 
for all $1\leq i_1<\cdots<i_k\leq q$.
\label{prop:vanishing}
\end{proposition}

\begin{proof} The image of $F$ is a quasi-affine variety, whose function field is isomorphic to $\C(F_1,\ldots,F_q)$. Then the dimension $p$ of $\im(F)$ is equal to the transcendence degree of $\C(F_1,\ldots,F_q)$ over \C. Without loss of generality, suppose $\{F_1,\ldots, F_p\}$ is a transcendence base for $\C(F_1,\ldots,F_q)$ over \C. Then the set $\{\d F_1, \ldots, \d F_p\}$ forms a basis for $\Omega(\C(F_1,\ldots,F_q))$, a vector space over $\C(F_1,\ldots,F_q), $ see \cite[Theorem 16.14]{Eis95}. Then 
\[
\d F_1 \we \cdots \we \d F_p\neq 0.
\] 
If $\{F_{i_1}, \ldots, F_{i_k}\} \subseteq \{F_1,\ldots, F_q\}$ with $k>p$, then $\{\d F_{i_1}, \ldots, \d F_{i_k}\}$ is linearly dependent over $\C(F_1,\ldots,F_q)$, hence 
\[
\d F_{i_1} \we \cdots \we \d F_{i_k}=0.
\]
\end{proof}

If $\Lambda=(\xi_1,\ldots, \xi_q)$ with $\xi_i\in \Z^{n+1}$ satisfying $\sum_{j=0}^n \xi_{ij}=0$, set
\[
\Phi_\Lambda:=(\Phi_{\xi_1}, \ldots , \Phi_{\xi_q}) \colon \P^\ell \rightarrowtail \C^q.
\]
Proposition~\ref{prop:vanishing} applies as follows.
\begin{corollary} Suppose $D$ is a rational singular subspace of $A^1$, and $\Phi_\Lambda$ is the rational mapping associated to an ordered integral basis $\Lambda$ of $D$. Then $\dim \Phi_\Lambda(U)$ is equal to the rank of $D$.
In particular, $\Phi_\Lambda(U)$ has positive codimension in $\C^q$. 
\label{cor:dim}
\end{corollary}

Let $[z_0:\cdots:z_q]$ be homogeneous coordinates on $\P^q$, and identify $\C^q$ with the $\P^q - \{z_0= 0\}$ as above. Let $\bar{Y}$ be the Zariski closure of $Y=\Phi_\Lambda(U)$ in $\P^q$. In homogeneous coordinates, $\Phi_\Lambda$ is given by
\[
\Phi_\Lambda = [1:\Phi_{\xi_1}:\cdots:\Phi_{\xi_q}] \colon \P^\ell \rightarrowtail \P^q.
\]
Clearing fractions, we have
\[
\Phi_\Lambda = [\Phi_{\nu_0}:\Phi_{\nu_1}:\cdots:\Phi_{\nu_q}]
\]
where  the master functions $\Phi_{\nu_i}$ are homogeneous polynomials of the same degree $d$. Moreover we may assume the $\Phi_{\nu_i}$ have no common factors. Equivalently,  the weights $\nu_i=(\nu_{i0}, \ldots, \nu_{in})$ are non-negative integer vectors with $\sum_{i=0}^n \nu_{ij} = d$ for $0\leq i \leq q$, whose supports have empty intersection. Then
\[
\Phi_{\xi_i}=\frac{\Phi_{\nu_i}}{\Phi_{\nu_0}},\quad \xi_i=\nu_i-\nu_0, \quad \text{and} \quad \omega_{\xi_i}= \omega_{\nu_i} - \omega_{\nu_0}.
\]


\begin{definition} $\Lambda$ is {\em essential} if 
every hyperplane $H \in \A$ appears as a component of $V(\Phi_{\nu_i})$ for some $i, \ 0\leq i\leq q$. 
\end{definition}
Henceforth we will tacitly assume $\Lambda$ is essential. We have $Y \subseteq \bar{Y} \cap (\C^*)^q$ in any case. If $\Lambda$ is not essential, the inclusion is proper, and may be proper otherwise -- see Example~\ref{ex:braid_v2}. 

The defining ideal $I=I_\Lambda$ of $\bar{Y}$ is generated by homogeneous polynomials $P(z_0,\ldots, z_q)$ for which $P(\Phi_{\nu_0},\ldots, \Phi_{\nu_q})$ vanishes identically on $U$,  or, equivalently, $P(1,\Phi_{\xi_1}, \ldots , \Phi_{\xi_q})=0$. We will sometimes refer to $I_\Lambda$ as the {\em syzygy ideal} of $\Lambda$.  The mapping $z_i \mapsto \Phi_{\nu_i}$ induces an isomorphism of rings
\[
\C[\bar{Y}] =\C[z_0,\ldots,z_q]/I \longrightarrow \C[\Phi_{\nu_0},\ldots,\Phi_{\nu_q}].
\]
In particular $\bar{Y}$ is irreducible. Identifying $\C[\bar{Y}]$ with $\C[\Phi_{\nu_0}, \ldots, \Phi_{\nu_q}]$, the dominant rational mapping $\Phi_\Lambda \colon \P^\ell \rightarrowtail \bar{Y}$ corresponds to the field extension 
\[
\C(\Phi_{\xi_1},\ldots, \Phi_{\xi_q}) \subseteq \C(x_1,\ldots,x_\ell).
\] 
The affine ring $\C[Y]$ is isomorphic to a localization of the ring of Laurent polynomials $\C[\Phi_{\nu_0}^{\pm 1}, \ldots, \Phi_{\nu_q}^{\pm 1}]$. 

If $\omega \in D-\{0\}$, we write $\omega=\omega_a=\sum_{i=0}^q a_i\omega_{\nu_i}=\sum_{i=1}^n a_i\omega_{\xi_i}$ with $a=(a_0,\ldots a_q) \in \C^{q+1} - \{0\}$, satisfying $\sum_{i=0}^q a_i=0$. Note that any $p$-fold wedge  product $\psi=\omega_{\xi_{i_1}} \we \cdots \we \omega_{\xi_{i_p}}$ is a cocycle for $\omega$. 

The one-form $\sum_{i=0}^q a_i \d \log (z_i)\in \Omega(\C^{q+1})$ is $\C^*$-invariant, and contracts trivially
along the Euler vector field, so it descends to a well-defined
rational one-form on $\P^q$.  This form restricts to a one-form in $\Omega(\bar{Y})$ which
we denote by $\tau_a$.  Since $Y \subseteq (\C^*)^q$, $\tau_a$ is regular on $Y$. Note that $\sum_{i=0}^q a_i \d \log (z_i) = \d \log \mu_a$, where $\mu_a=\prod_{i=0}^q z_i^{a_i}$ is a master function for the arrangement of coordinate hyperplanes in $\P^q$. 

We show that the zeros of $\tau_a$ pull back to zeros of $\omega_a$.

\begin{lemma} Let $x \in U$ and $y=\Phi_\Lambda(x) \in Y$. Then $\omega_a(x)=0$ if and only if $\tau_a(y) \in \ker(\Phi_\Lambda^*)_y$.
\label{lem:kernel}
\end{lemma}

\begin{proof} We have 
\[
\omega_a=\sum_{i=0}^q a_i \d \log \Phi_{\nu_i} 
 = \d \log \prod_{i=0}^q \Phi_{\nu_i}^{a_i},
= \Phi_\Lambda^*(\d \log \prod_{i=0}^q z_i^{a_i})
=\Phi_\Lambda^*(\tau_a).
\]
The result follows upon localization at $y$.
\end{proof}

Note that 
\[
\Phi_\Lambda^*(\tau_a)=\sum_{i=0}^q \sum_{j=0}^{\ell} \frac{a_i}{\Phi_{\nu_i}} \frac{\partial \Phi_{\nu_i}}{\partial x_j} \d x_j.
\] Then $\tau_a(y) \in \ker (\Phi_\Lambda^*)_y$ if and only if $\begin{bmatrix}\frac{a_0}{y_0} & \cdots & \frac{a_q}{y_q} \end{bmatrix}$ lies in the left null space of the Jacobian of $\Phi_\Lambda$. 

Let $\Sing(\Phi_\Lambda)$ denote the singular locus of $\Phi_\Lambda$, and $\Sing(\bar{Y})$ the singular locus of $\bar{Y}$. Let $S_\Lambda=\Sing(\Phi_\Lambda) \cup \Phi_\Lambda^{-1}(\Sing(\bar{Y})) \subseteq U$.
\begin{theorem} $V(\omega_a)$ contains $\Phi_\Lambda^{-1}(V(\tau_a))$, and $V(\omega_a) - \Phi^{-1}(V(\tau_a))$ is a subset of $S_\Lambda$.
\label{thm:crit_fixed}
\end{theorem}

\begin{proof} The first statement is immediate from Lemma~\ref{lem:kernel}. If $\Phi_\Lambda$ is nonsingular at $x \in U$ then the Jacobian of $\Phi_\Lambda$ attains its maximal rank $p=\dim(Y)$ at $x$.  If in addition $\bar{Y}$ is nonsingular at $y=\Phi_\Lambda(x)$, then $\dim(\Omega(Y)_y)=p$. Then $\ker(\Phi_\Lambda^*)_y=0$.  The second statement then follows from Lemma \ref{lem:kernel}.
\end{proof}

\begin{corollary} 
Suppose $D$ is a rational singular subspace of $A^1$, with integral basis $\Lambda$. Let $p$ be the rank of $D$. If $\d \log(\Phi) \in D$ then $\crit(\Phi)\subseteq S_\Lambda$ or  $\codim (\crit(\Phi))\leq p$. 
\label{cor:codimp}
\end{corollary}

\begin{proof} 
Write $\omega = \d \log(\Phi) = \sum_{i=1}^q a_i \omega_{\xi_i}$. The hypothesis implies $\dim Y=p$, so $V(\tau_a)$ is empty or has codimension at most $p$ in $Y$. In the first case $V(\omega)\subseteq S_\Lambda$ by the preceding theorem. Otherwise, $V(\omega)\supseteq \Phi_\Lambda^{-1}(V(\tau_a))$ has codimension at most $p$.
\end{proof}

\begin{corollary} 
Suppose $D$ is a rational singular subspace of $A^1$, with integral basis $\Lambda$.  If $\d \log(\Phi) \in D$, then $\crit(\Phi) - S_\Lambda$ is a union of fibers of $\Phi_\Lambda$. 
\label{cor:fibers}
\end{corollary}
The fibers of $\Phi_\Lambda$ are intersections of level sets of the rational master functions $\Phi_{\xi_i}$, for $1\leq i\leq q$.

We have not used the assumption that $\{\xi_1,\ldots, \xi_q\}$ is linearly independent, i.e., that the dimension of $D$ is strictly greater than $p$. This hypothesis rules out a trivial case.
\begin{proposition} Suppose $a\neq 0$. Then $\tau_a$ is not identically zero on $Y$. 
\label{prop:trivial}
\end{proposition}

\begin{proof} 
If $\tau_a$ is zero on $Y$, then 
\[
\Phi_\Lambda^*(\tau_a)=\sum_{i=0}^q a_i \d \log(\Phi_{\nu_i})
=\sum_{i=0}^q a_i \omega_{\nu_i}
=\sum_{i=1}^q a_i\omega_{\xi_i}
\]
is zero on $U$.  This contradicts the assumption that $\{\omega_{\xi_1},\ldots, \omega_{\xi_q}\}$ is a basis for $D$.
\end{proof}

A singular subspace of rank 1 is called an {\em isotropic} subspace of $A^1$.
\begin{corollary} Suppose $\Lambda$ is an integral basis of an isotropic subspace of $A^1$. Then 
\begin{enumerate}
\item if $\omega=\d \log(\Phi) \in D$ then $\Sing(\Phi_\Lambda)\subseteq \crit(\Phi)$. 
\item if $0\neq\omega=\d \log(\Phi) \in D$, then the components of $\crit(\Phi)- S_\Lambda$ are disjoint hypersurfaces in $U$.
\end{enumerate}
\label{cor:rankonesing}
\end{corollary}

\begin{proof} Since $\dim(Y)=1$, the Jacobian of $\Phi_\Lambda$ vanishes identically at points of  $\Sing(\Phi_\Lambda)$. Then $\Sing(\Phi_\Lambda) \subseteq  V(\omega)=\crit(\Phi)$ by Lemma~\ref{lem:kernel}. If $a\neq 0$, then $\tau_a$ doesn't vanish identically on $Y$ by Proposition~\ref{prop:trivial}. Then $V(\tau_a)$ is zero-dimensional. Assertion (ii) follows from Theorem~\ref{thm:crit_fixed}.
\end{proof}

Corollary~\ref{cor:rankonesing}(i) can be used to locate singular fibers in \Ceva pencils \cite[Def.~4.5]{FY07} -- see Example~\ref{ex:referee}.

\subsection{Zeros of $\tau_a$} 
We apply the method of Lagrange multipliers to find the zeros of $\tau_a$. The 
argument applies even if $\bar{Y}$ is singular.
Fix a set of homogeneous generators $\{P_1,\ldots, P_r\}$ of the defining ideal $I=I_\Lambda\subseteq S$ of $\bar{Y}$. Write $\partial_j$ for $\frac{\partial\phantom{z_j}}{\partial z_j}$. Let 
\[
\Jac_\Lambda=\begin{bmatrix}\partial_j P_i \end{bmatrix}
\]
be the Jacobian of $(P_1,\ldots, P_r)$. The rank of $\Jac_\Lambda$ at a nonsingular point $y \in \bar{Y}$ is equal to $q-p$, the codimension of $\bar{Y}$.

\begin{lemma} Let $y\in Y$. Then 
$y \in V(\tau_a)$ if and only if 
\(
\begin{bmatrix} \frac{a_0}{y_0} \cdots \frac{a_q}{y_q} \end{bmatrix}
\)
is an element of the row space of $\Jac_\Lambda(y)$.
\label{lem:rowspace}
\end{lemma}

\begin{proof} 
The one-form $d \log \mu_a =\sum_{i=0}^q a_i \d \log(z_i)\in \Omega(\P^q)$ restricts to $\tau_a \in \Omega(Y)$. There is an exact sequence of $\C[\bar{Y}]$-modules
\[
I/I^2  \xrightarrow{d}  \C[\bar{Y}]\otimes_{\C[\P^q]}\Omega[\P^q]  \to \Omega[\bar{Y}] \to 0,
\]
where
$d$ is given by right multiplication by $\Jac_\Lambda$  \cite[Sec. 16.1]{Eis95}. Localization at the maximal ideal corresponding to $y$ preserves exactness of this sequence, so $\tau_a(y) \in \Omega(Y)_y$ vanishes if and only if $\tau_a(y)$ is in the image of $d$.
\end{proof}

For each $i\geq 0$, let $\Fitt_i(a,I)$ be the variety in $\P^q$ defined by the $(q+1-i) \times (q+1-i)$ minors of the $(r+1)\times (q+1)$ matrix 
\begin{equation}
\begin{bmatrix}
a_0/z_0 & \cdots & a_q/z_q \\
\partial_0 P_1 & \cdots & \partial_q P_1\\
\vdots & \ddots & \vdots \\
\partial_0 P_r & \cdots & \partial_q P_r
\end{bmatrix}.
\label{form:matrix}
\end{equation}
The ideal $\Fitt_i(a,I)$ is independent of the choice of generating set for $I$. Similarly, let $\Fitt_i(\Jac_\Lambda)$ denote the variety in $\P^q$ defined by the $(q+1-i) \times (q+1-i)$ minors of $\Jac_\Lambda$. Let  $\bar{Y}_{\rm reg} = \bar{Y} - \Sing{\bar Y}$. Then $\bar{Y}_{\rm reg} = \bar{Y} \cap \Fitt_p(\Jac_\Lambda)-\Fitt_{p+1}(\Jac_\Lambda)$, where $p=\dim \bar{Y}$. 


If $Y$ is smooth, then $V(\tau_a)$ is a Fitting variety. More generally:

\begin{corollary} $V(\tau_a) \cap \bar{Y}_{\rm reg} = \Fitt_p(a,I) \cap  \bar{Y}_{\rm reg}$.
\label{cor:smooth}
\end{corollary}

\begin{proof} At any point of $\bar{Y}_{\rm reg}$ the rank of $\Jac_\Lambda$ is equal to $q-p=\codim \bar{Y}$, so the rank of matrix \eqref{form:matrix} is at least $q-p$. Then $\begin{bmatrix} \frac{a_0}{y_0} & \cdots & \frac{a_q}{y_q} \end{bmatrix}$ is in the row space of $\Jac(y)$ if and only if \eqref{form:matrix} has rank $q-p$, if and only if all $(q-p+1)\times (q-p+1)$ minors vanish.
\end{proof}

\begin{remark}
In fact, the intersections $\bar{Y} \cap \Fitt_i(\Jac_\Lambda), \ p\leq i\leq q$ determine a stratification of $\bar{Y}$ by locally-closed subvarieties, and $V(\tau_a)$ coincides with $\Fitt_i(a,I)$ on the stratum $\Fitt_i(\Jac_\Lambda)- \Fitt_{i+1}(\Jac_\Lambda)$, by the same argument.
\label{rmk:strata}
\end{remark}

In view of Proposition~\ref{prop:zeros} we study the zeros of cocycles for $\tau_a$. Since $\dim Y=p$, every $\psi \in \Omega^p(Y)$ is a cocycle for $\tau_a$.  For $0\leq i\leq q$, set $\tau_i=\frac{\d y_i}{y_i}$, so that $\tau_a=\sum_{i=0}^q a_i \tau_i = \sum_{i=1}^q a_i (\tau_i - \tau_0)=\sum_{i=1}^q \d \log(y_i/y_0)$. 

\begin{proposition} The intersection 
\[
\bigcap \{V(\xi) \mid \xi \in \Omega_\C^p(Y), \tau_a\we \xi=0\}
\]
is contained in $\Sing(\bar{Y})$.
\label{prop:nozeros}
\end{proposition}

\begin{proof} For $I=\{i_1< \cdots <i_p\}_< \subseteq \{1,\ldots, q\}$, consider the $p$-form 
\[
\xi_I=(\tau_{i_1}-\tau_0) \we \cdots  \we (\tau_{i_p}-\tau_0).
\] 
Then $\tau_a \we \xi_I=0$. But $\dim \bar{Y}=p$, so at each point of $\bar{Y}_{\rm reg}$, $\xi_I$ must be nonzero for some $I$.
\end{proof}



\subsection{The case $q=p+1$} Suppose $D$ is a singular subspace of rank $\dim(D)-1$, with integral basis $\Lambda=\{\omega_{\xi_1}, \ldots ,\omega_{\xi_q}\}$. Then $\bar{Y}$ is defined by a single homogeneous polynomial $P(z_0,\ldots, z_q)$. This hypothesis holds for all the examples we know, with one exception: the Hessian arrangement, which supports a rational singular subspace of dimension three and rank one -- see Examples~\ref{ex:hessian} and ~\ref{ex:hessian2}. 

%
%
%
%

Consider the rational mapping
\begin{equation}
\rho=[z_0\partial_0 P: \cdots: z_q\partial_q P] \colon \P^q \rightarrowtail \P^q.
\label{form:rho}
\end{equation}
This map has poles along $\Sing(\bar{Y})$ and $\Sing(\bar{Y} \cap \C^I)$ where $\C^I$ is the coordinate subspace $z_i=0, \, i \in I$.  It is regular on $\bar{Y}_{\rm reg} \cap (\C^*)^q$. By Euler's formula, $\sum_{j=0}^q z_j \partial_j P=\deg(P) P$, so the image of $\bar{Y}$ under $\rho$ is contained in the hyperplane $\sum_{j=0}^q z_j=0$.

\begin{proposition} Suppose $\bar{Y}$ is the hypersurface given by $P=0$, and $\rho$ is as given in \eqref{form:rho}. Then 
\begin{enumerate}
\item $V(\tau_a)\neq \emptyset$ if and only if $a\in \rho(Y)$. 
\item If $a \in \rho(Y)$ then $V(\tau_a)=\rho^{-1}(a)$. 
\item For generic $a \in \rho(Y)$, $V(\tau_a)$ is nonempty and every component has codimension equal to $\dim \rho(\bar{Y})$.
\end{enumerate}
\label{prop:fibers}
\end{proposition}

\begin{proof} The first two assertions follow from Lemma~\ref{lem:rowspace}, and the third follows from the second. 
\end{proof}

\begin{corollary} Suppose
\begin{enumerate}
\item $\Phi_\Lambda$ is nonsingular on $U$,
\item $\bar{Y}$ is nonsingular and intersects the coordinate hyperplanes transversely,
\item $\dim \rho(\bar{Y})=\dim \bar{Y}$, and
\item $\Phi_\Lambda(U)=\bar{Y}\cap(\C^*)^q$.  
\end{enumerate}
If $\d \log(\Phi)=\sum_{i=1}^q a_i\omega_{\xi_i}$ and $a_i\neq 0$ for all $i$, then $\crit(\Phi)$ is nonempty and every component has codimension $q-1$.
\label{cor:birat}
\end{corollary}

\begin{proof} We have observed that $\rho(\bar{Y})$ is contained in the hyperplane $\Delta$ defined by $\sum_{i=0}^q z_i=0$.  The second hypothesis ensures that $\Sing(Y \cap \C^I)$ is empty for every $I\subseteq \{0,\ldots, q\}$. Then  $\rho$ is regular on $\bar{Y}$. By the third condition, $\dim \rho(\bar{Y})=q-1=\dim \Delta$. Since $\Delta$ is irreducible, we conclude $\rho(\bar{Y})=\Delta$. Since $Y=\bar{Y} \cap (\C^*)^q$ by (iii), $\Delta \cap (\C^*)^q =\rho(Y)$. The result then follows from Proposition~\ref{prop:fibers} and Theorem~\ref{thm:crit_fixed}, since (i) and (ii) imply $S_\Lambda=\emptyset$ and the fibers of $\Phi_\Lambda$ all have codimension $q-1$.
\end{proof}

The hypotheses in Corollary~\ref{cor:birat} are satisfied in many examples. The third condition is automatic in case $q=2$. The last two conditions hold in every example we know where (i) and (ii) hold.

\subsection{Examples}
\begin{example}[Example~\ref{ex:braid}, continued]\label{ex:braid_v2}
Let $D$ be the rational singular subspace of $A^1 \cong \C^5$ with basis $\Lambda=\{\omega_{010}-\omega_{100}, \ \omega_{001}-\omega_{100}\}$. 

We have 
\[
\Phi_\Lambda =[\Phi_{100}:\Phi_{010}:\Phi_{001}]=[x(y-z):y(x-z):z(x-y)].
\]
$\Phi_\Lambda$ is nonsingular on $U$, and $\bar{Y}=Y \cap (\C^*)^5$. The components of $\Phi_\Lambda$ satisfy the homogeneous relation $\Phi_{100}-\Phi_{010}+\Phi_{001}=0$, so $\bar{Y}$ is the line $z_0-z_1+z_2=0$ in $\P^2$. Corollary~\ref{cor:birat} implies $\crit(\Phi_{abc})$ is nonempty if and only if  $a+b+c=0$ and $a, b$ and $c$ are nonzero. In this case,  
\[
\crit(\Phi_{abc})=(\rho\circ \Phi_\Lambda)^{-1}([a:b:c])=\Phi_\Lambda^{-1}([a:-b:c]),
\] that is, $\crit(\Phi_{abc})$ 
is given by 
\begin{align*}
[x(y-z):y(x-z):z(x-y)]&=[a:-b:c], \ \text{or, equivalently}\\
[x(y-z):y(z-x)]&=[a:b]
\end{align*}
as we found earlier.
It has codimension one in $\P^2$. In this example the map $\Phi_\Lambda$ has connected generic fiber, hence $\crit(\Phi_{abc})$ is connected. If $a, b$, or $c$ is zero, and $a+b+c=0$, then $\crit(\Phi_{abc})$ is empty.

The basis $\Lambda$ above has special properties that resulted in the linear syzygy of master functions: we will revisit this in the next section. By way of comparison, consider the basis $\Lambda'=\{\omega_{120}-\omega_{012},\omega_{300}-\omega_{012}\}$. 
%
Then 
\[
\begin{split}
\Phi_{\Lambda'}&=[\Phi_{010}\Phi_{001}^2:\Phi_{100}\Phi_{010}^2:\Phi_{100}^3] \\
&=[yz^2(x-y)^2(x-z):xy^2(x-z)^2(y-z):x^3(y-z)^3].
\end{split}
\]
A {\em Macaulay~2} calculation \cite{M2} shows that $\Phi_{\Lambda'}$ is nonsingular on $U$. Using the identity $\Phi_{100}-\Phi_{010}+\Phi_{001}=0$, one finds that the Zariski closure $\bar{Y'}$ of $Y'=\Phi_{\Lambda'}(U)$ is defined by
\[
z_1^3-z_0^2z_2-4z_0z_1z_2-2z_1^2z_2+z_1z_2^2=0.
\]
Then $\bar{Y'}$ is an irreducible cubic with a node at $[z_0:z_1:z_2]=[-2:1:-1]$. This is a point of $Y'$. It is also in $\tilde{\Phi}_{\Lambda'}(E)$, where $\tilde{\Phi}_{\Lambda'}$ is the lift of $\Phi_{\Lambda'}$ to the blow-up of $\P^2$ at the four base points, and $E$ is the exceptional divisor over $[0:1:0]$.

The image of $Y'$ under $\rho$ misses the three points $[0:-1:1]$, $[-2:1:1]$, and $[-2:3:-1]$, corresponding to the three one-forms $\omega_{100}- \omega_{010}, \ \omega_{100}- \omega_{001}$, and $\omega_{010}- \omega_{001}$ in $D$ that have empty zero locus. In particular $Y' \neq \bar{Y'}\cap (\C^*)^5$.
\end{example}
\begin{example} 
\label{ex:referee}
Let \A\ be the arrangement with defining equation $Q=Q_0Q_1Q_2$ where 
\[
\begin{aligned}
Q_0&=(x+z)(2x-y-z)(2x+y-z)\\
Q_1&=(x-z)(2x+y+z)(2x-y+z), \ \ \text{and}\\
Q_2&=(y+z)(y-z)z
\end{aligned}
\]

The image of $\Phi_\Lambda=[Q_0:Q_1:Q_2] \colon \P^2 \rightarrowtail \P^2$ is the line $z_0-z_1+2z_2=0$. The fibers are the cubics passing through nine points, three on each of three concurrent lines -- \A\ is a specialization of the Pappus arrangement. One of these cubics is $x(4x^2-y^2-3z^2)=0$. Although $\bar{Y}$ is smooth, $\Phi_\Lambda$ is singular at two points of $U$, given by $x=y^2+3z^2=0$. These two points lie in $\crit(Q_0^aQ_1^bQ_2^c)$ for every $a,b,c$ with $a+b+c=0$.

Similarly, if \A\ the subarrangement of the Hessian arrangement (\ref{ex:hessian}) defined by $P_1P_2P_3=0$, then every critical set $\crit(P_1^aP_2^bP_3^c)$, $a+b+c=0$, contains the three points $[1:0:0]$, $[0:1:0]$, $[0:0:1]$ of $U$ where the fourth special fiber $xyz=0$ is  singular. The map $\Phi_\Lambda=[P_1:P_2:P_3]$ is singular at these points, although $\bar{Y}$ is smooth, given by $\zeta P_1+\zeta^2 P_2 + P_3=0$. (See also Example~\ref{ex:hessian2}.)  
\end{example}
\begin{remark}
In fact, Corollary~\ref{cor:rankonesing} can be used to detect \Ceva pencils \cite{FY07} (see ~\ref{subsec:ceva}). For instance the master function 
\[
\Phi = \frac{x(y-z)}{y(x-z)}
\]
 has critical points $[0:1:0]$ and $[1:1:0]$, the singular points of the third completely decomposable fiber in Example~\ref{ex:braid}.
\end{remark}
\begin{example}[Example~\ref{ex:nondecomp}, continued]
In this example,
the one-form $\omega$ has no decomposable $2$-cocycle. Indeed there are no singular subspaces of $A^1$ of rank $p=1$ or $p=2$. For $p=1$ this holds because \A\ is 2-generic. For $p=2$ one can verify the statement computationally using the approach of \cite{LiFiSch09b}; in \cite{BFW11} we give a combinatorial argument based on Theorem~\ref{thm:rank}. 
Setting
\[
\Psi=[{x\over y+z+w}:{y\over x+z+w}:
{z\over x+y+w}:{w\over x+y+z}] \colon \P^3 \rightarrowtail \P^3,
\]
we have $\omega=\d \log(\Phi) =\Psi^*(\tau)$ where 
\[
\tau=a\, \d \log(y_0)+b\, \d \log(y_1)+c\, \d \log(y_2)+d\, \d \log(y_3).
\]
The map $\Psi$ is dominant. The one-dimensional critical locus of $\Phi$ is a fiber of a different map 
\[
[x(y+z+w):y(x+z+w):z(x+y+w):w(x+y+z)].
\]
\end{example}
 \end{section}

\begin{section}{Linear dependence among master functions} 
In this section we consider a singular subspace $D$ with an integral basis $\Lambda$ such that $Y_\Lambda$ is linear, i.e., the syzygy ideal $I_\Lambda$ is generated by homogeneous linear forms. We saw this phenomenon 
in Example~\ref{ex:braid_v2}.
We start with a trivial example that will be useful for what follows.

\subsection{Example: equations for the critical locus}
\label{subsec:trivial}
Suppose \A\ is an essential arrangement of $n+1$ hyperplanes in $\P^\ell$, with $n>\ell$. Then $D=A^1$ is a singular subspace, with integral basis $\{\omega_i - \omega_0 \mid 1\leq i\leq n\}$. The corresponding rational mapping is 
\[
\Phi=\Phi_\Lambda = [\alpha_0: \cdots \colon \alpha_n] \colon \P^\ell \rightarrowtail \P^n,
\]
and $\bar{Y}=\bar{Y}_\Lambda$ is a linear subvariety of $\P^n$. The reader will recognize that this is the usual identification of a labelled vector configuration, $(\alpha_0,\ldots,\alpha_n)$, with a point in the Grassmannian of $\ell$-planes in $\P^n$. Let us denote $\bar{Y}_\Lambda$ by $L_\A$. (This is an abuse of notation; $\bar{Y}_\Lambda$ depends on the choice of defining forms $\alpha_i$.)

Most of the results of the previous section are vacuous in this situation, but Lemma~\ref{lem:rowspace} tells us something:
\begin{theorem} Let $B=\begin{bmatrix} b_{ij} \end{bmatrix}$ be an $(\ell+1) \times (n+1)$ matrix such that $L_\A$ is the kernel of $B$. Then, for any $\lambda \in \C^n$, the critical locus of $\Phi_\la$ is defined by the $\binom{n+1}{\ell+1}$ equations  
\[
\sum_{i\in I} \sigma(i,I)b_{I - \{i\}} \frac{\la_i}{\alpha_i(x)} = 0
\]
where $I$ ranges over the subsets of $\{0,\ldots, n\}$ of size $\ell+1$, the coefficient $b_J$ is given by $b_J=\det \begin{bmatrix} b_{ij} \mid  j\in J\end{bmatrix}$, and $\sigma(i,I)=\pm 1$, depending on the position of $i$ in $I$. 
\label{thm:recip}
\end{theorem}

\begin{proof} The linear forms defined by the rows of $B$ generate the syzygy ideal $I_\Lambda$. Then the Jacobian $\Jac_\Lambda$ is equal to $B$.  With the observation that $S_\Lambda=\emptyset$, setting $a=\lambda$ in Lemma~\ref{lem:rowspace} and applying Theorem~\ref{thm:crit_fixed} yields the claim.
\end{proof}

The columns of the matrix $B$ above define a realization of the matroid dual to the matroid of \A. In Example~\ref{ex:braid}, 
\[
B= \begin{bmatrix}0&-1&1&0&0&1 \\ -1&0&1&0&1&0 \\ -1&1&0&1&0&0
\end{bmatrix}.
\]
(The matroid of \A\ is self-dual.) Theorem~\ref{thm:recip} says that $\crit(\Phi_{abc})$ is defined by the $4 \times 4$ minors of 
\[
\begin{bmatrix}
\frac{a}{x}&\frac{b}{y}&\frac{c}{z}&\frac{c}{x-y}&\frac{b}{x-z}&\frac{a}{y-z}\\
0&-1&1&0&0&1 \\ -1&0&1&0&1&0 \\ -1&1&0&1&0&0
\end{bmatrix}.
\]
These 15 equations reduce to the single equation found in Example~\ref{ex:braid}.

\subsection{Linear hypersurfaces} Suppose $\bar{Y}$ is a linear and $p=\rank(D)=q-1$, i.e., $\bar{Y}$ is a linear hyperplane in $\P^q$. Let $P(z)=\sum_{j=0}^q b_j z_j$ be a generator for the syzygy ideal. It is no loss to assume $b_j\neq 0$ for all $j$, or equivalently, $\bar{Y}$ is not contained in any coordinate hyperplane. Otherwise some proper subset of $\{\Phi_{\nu_0}, \ldots, \Phi_{\nu_q}\}$ is linearly dependent. 

\begin{proposition}
Suppose $\bar{Y}$ is a hyperplane not contained in any coordinate hyperplane in $\P^q$.  Let $\la=\sum_{i=0}^q a_i \nu_i$. Then for generic $a$ satisfying $\sum_{i=0}^q a_i=0$, $\crit(\Phi_\la)-S_\Lambda$ is nonempty and every component of $\crit(\Phi_\la)-S_\Lambda$ has codimension equal to the rank of $D$.
\label{prop:linhyp}
\end{proposition}

\begin{proof} The syzygy ideal $I_\Lambda$ is generated by a linear polynomial $P(z)=\sum_{j=0}^q b_jz_j$, and $b_j\neq 0$ for $0\leq j\leq q$ by hypothesis. 
The map $\rho=[b_0z_0:\cdots :b_qz_q]$ of \eqref{form:rho} is an automorphism of $\P^q$ since all of the $b_j$ are nonzero.
Consequently, $\rho$ maps the hyperplane $\bar{Y}$ isomorphically to the hyperplane $\Delta$ defined by $\sum_{i=0}^q z_i=0$.  The result then follows from Proposition~\ref{prop:fibers}.
\end{proof}


\subsection{The general case} Suppose the singular subspace $D\subseteq A^1$ has an integral basis $\Lambda$ for which $\bar{Y}=\bar{Y}_\Lambda$ is a linear variety in $\P^q$. Choose a linear isomorphism $\phi : \P^p \to \bar{Y}$, given by a $(q+1) \times (p+1)$ matrix $B=\begin{bmatrix} b_{ij} \end{bmatrix}$. Assume $D$ is not contained in any coordinate hyperplane. Then the intersections of $\bar{Y}$ with the coordinate hyperplanes in $\P^q$ determine an essential arrangement \B\ of $q+1$ not necessarily distinct hyperplanes in $\P^p$, with defining forms $\beta_i(x)=\sum_{i=0}^p b_{ij}x_j$, for $0\leq i\leq q$. By construction, the subspace $L_\B$ of $\P^q$, as described in section \ref{subsec:trivial}, is equal to $\bar{Y}$. 

\begin{theorem} 
Suppose $\bar{Y}_\Lambda=L_\B$ is a linear subspace not contained in any coordinate hyperplane in $\P^q$, and $Y=\bar{Y}\cap (\C^*)^q$. Let $\la=\sum_{i=0}^q a_i \nu_i$ with $\sum_{i=0}^q a_i=0$.  Let $\Psi_a = \prod_{i=0}^q \beta_i^{a_i}$ be the master function on the complement of the arrangement \B\ in $\P^p$ corresponding to $a$. Then 
\begin{enumerate}
\item $\crit(\Phi_\la)-S_\Lambda\neq \emptyset$ if and only if $\crit(\Psi_a) \neq \emptyset$, 
\item $\crit(\Phi_\la)-S_\Lambda =\Phi_\Lambda^{-1}(\rho(\crit(\Psi_a)))$, and
\item $\codim (\crit(\Phi_\la)-S_\Lambda) \leq \codim \crit(\Psi_a)$.
\end{enumerate}
\end{theorem}

\begin{proof} 
Just as in section~\ref{subsec:trivial}, the one-form $\tau_a$ on $\bar{Y}$ pulls back to $\d \log(\Psi_a)$ under the isomorphism $\phi$,  and the assertions follow from Proposition~\ref{prop:fibers}. 
\end{proof}

\begin{example}[Example~\ref{ex:braid}, continued]
\label{ex:line}
We saw in Example~\ref{ex:braid_v2} that the variety $\bar{Y}$ is given by $z_0-z_1+z_2=0$ in $\P^2$, and $Y=\bar{Y} \cap (\C^*)^2$. We can take $\phi \colon \P^1 \xlongrightarrow{\cong} \bar{Y}$ to be given by the matrix 
\[
B=\begin{bmatrix}
1&0\\
1&1\\
0&1
\end{bmatrix}
\]
Then associated arrangement \B\ consists of three points $[1:0]$, $[1:1]$, $[0:1]$ in $\P^1$. The complement of \B\ has Euler characteristic $-1$, so a generic \B-master function $\Psi_{abc}$ has a single nondegenerate critical point. A computation 
shows that this holds if $a, b$, and $c$ are nonzero. We reach the same conclusion as before, that $\crit(\Phi_{abc})$ has codimension one.

\end{example}

\subsection{Multinets and codimension-one critical sets}\label{subsec:ceva}
Next we use the main result of \cite{FY07} to give a complete description of $\crit(\Phi_\la)$ for any $\omega=\d \log(\Phi)\in \R^1(\A)$. As we observed earlier, if $\omega \in \R^1(\A)$, then $\omega$ has a nontrivial decomposable 1-cocycle $\psi$. The statement that $\omega \we \psi = 0$ means, for each $x \in U, \{\omega(x),\psi(x)\}$ 
is linearly dependent. Then there are functions $a$ and $b$ on $U$ such that $a(x)\omega(x) + b(x)\psi(x)=0$ for all $x \in U$. This implies $V(\omega)$ contains $V(b)-V(a)$, which is a hypersurface unless it is empty. It was this observation that led to the current research.

According to \cite{LY00}, the maximal isotropic subspaces $D$ of $A^1$ of dimension at least two are the components of $\R^1(\A)$, and they intersect trivially. By \cite[Theorem 3.11]{FY07}, such a component has an integral basis $\Lambda=(\omega_{\xi_1}, \ldots, \omega_{\xi_q})$, with the property that the corresponding polynomial master functions $\Phi_{\nu_0}, \ldots, \Phi_{\nu_q}$ are all collinear in the space of degree $d$ polynomials. Then $\bar{Y}=\bar{Y}_\Lambda$ is a line in $\P^q$. The homogenized basis $\{\omega_{\nu_0}, \ldots, \omega_{\nu_q}\}$ corresponds to the characteristic vectors $\nu_i$ of the blocks in a multinet structure on a subarrangement of \A, as defined below. 

For $X\subseteq \P^\ell$, write $\A_X=\{H \in \A \mid X \subseteq H\}$. A {\em rank-two flat} of \A\ is a subspace $X$ of the form $H\cap K$ for some $H, K \in \A, H \neq K$. If ${\mathcal P}$ is a partition of \A, the {\em base locus} of ${\mathcal P}$ is the set of rank-two flats of \A\ obtained by intersecting hyperplanes from different blocks of ${\mathcal P}$.

\begin{definition} A $(q+1,d)$-multinet on \A\ is a pair $({\mathcal P},m)$ where ${\mathcal P}$ is a partition $\{\A_0,\ldots, \A_q\}$ of \A\ into $q+1$ blocks, with base locus \X,  and $m\colon \A \to \Z_{>0}$ is a multiplicity function, satisfying
\begin{enumerate}
\item $\sum_{H \in \A_i} m(H) =d$ for every $i$.
\item For each $X\in \X$, $\sum_{H \in \A_i\cap \A_X} m(H) =n_X$ for some integer $n_X$, independent of $i$.
\item For each $i$, $\bigcup \A_i - \bigcup \X$ is connected.
\end{enumerate}
\end{definition}

The third condition says that ${\mathcal P}$ cannot be refined to a $(q',d)$-multinet with the same multiplicity function, with $q'>q+1$. Given a multinet on \A, let $\nu_i=\sum_{H_i \in \A_i} m(H_i)e_i$, for $0\leq i\leq q$, and $\xi_i=\nu_i - \nu_0$ for $1\leq i\leq q$. We call $\nu_0, \ldots, \nu_q$ the {\em characteristic vectors} of $({\mathcal P},m)$. We have the following results from \cite{FY07}.

\begin{theorem}[{\rm \cite[Corollary 3.12]{FY07}}] Suppose $D$ is a maximal isotropic subspace of $A^1$ of dimension $q\geq 2$. Then there is a subarrangement $\A'$ of \A\ and a $(q+1,d)$-multinet on $\A'$ whose characteristic vectors $\nu_0, \ldots, \nu_q$ yield an integral basis $\omega_{\xi_1}, \ldots, \omega_{\xi_q}$ of $D$.
\end{theorem}

\begin{theorem}[{\rm \cite[Theorem 3.11]{FY07}}] Suppose $\nu_0, \ldots, \nu_q$ are the characteristic vectors of a $(q+1,d)$-multinet structure on \A. Then each of the master functions $\Phi_{\nu_i}$, $i\geq 2$, is a linear combination of $\Phi_{\nu_0}$ and $\Phi_{\nu_1}$. Moreover every fiber of the mapping $[\Phi_{\nu_0}:\Phi_{\nu_1}] \colon \P^\ell \rightarrowtail \P^1$ is connected. 
\label{thm:connected}
\end{theorem}

Given this result, the analysis of critical sets proceeds exactly as in the Example~\ref{ex:line}. First, we need a lemma about isolated critical points.

\begin{lemma}\label{lem:isolated}
Suppose $\A$ is an affine arrangement of $n$ hyperplanes in $\C^\ell$, and $W$ is a nonempty Zariski-open
subset of $\C^\ell$.   Then there is a 
nonempty Zariski-open subset $L$ of $\C^n$
such that $W\cap \crit(\Phi_\lambda)$ consists of $\abs{\chi(U)}$ points,
for each $\lambda\in L$.
\end{lemma}
\begin{proof}
By \cite[Theorem~1.1]{OT95}, there is a nonempty Zariski-open subset $L'$ of $\C^n$ for which $\crit(\Phi_\lambda)$ is isolated and consists of $\abs{\chi(U)}$ points,
for $\lambda\in L'$.  Let 
\[
\Sigma=\set{(\lambda,v)
\in\C^n\times U\colon \omega_\lambda(v)=0},
\]
an $n$-dimensional smooth complex variety by \cite[Prop.~4.1]{OT95}.  
Let $\pi_i$ for $i=1,2$ denote its projections onto $\C^n$ and $U$, 
respectively.  Then $\pi_2^{-1}(U\cap W)$ and $\pi_1^{-1}(L')$ are each nonempty Zariski-open subsets of $\Sigma$, as is their intersection $Z$. Then $\pi_1(Z)$ is a finite union of locally-closed subsets of $\C^n$ -- see \cite[Exercise 3.19]{Har77}. Since $Z$ is dense in $\Sigma$, $\pi_1(Z)$ is dense in $\C^n$. Hence $\pi_1(Z)$ contains a Zariski-open subset $L$, which has the required property.
\end{proof}


\begin{theorem} Suppose $\omega=\d \log(\Phi) \in \R^1(\A)$, and $D$ is the maximal isotropic subspace of $A^1$ containing $\omega$. Let $\Lambda$ be the integral basis of $D$ arising from the associated multinet. Then 
\begin{enumerate}
\item For every $\omega \in D$, $\Sing(\Phi_\Lambda) \subseteq \crit(\Phi)$, and, 
\item For generic $\omega\in D$, $\crit(\Phi)-\Sing(\Phi_\Lambda)$ is a union of $\dim(D)-1$ connected smooth hypersurfaces of the same degree.
\end{enumerate}
\end{theorem}

\begin{proof}
Write $q=\dim(D)$ and let $\nu_0, \ldots, \nu_q$ be the characteristic vectors of the $(q+1,d)$-multinet corresponding to $D$.
Write $\omega=\sum_{i=0}^q a_i \omega_{\nu_i}$. By the preceding theorem, for each $2 \leq k\leq q$, there is a linear relation $\Phi_{\nu_k} = b_k \Phi_{\nu_0} + c_k\Phi_{\nu_1}$. Then $\bar{Y}$ is a line in $\P^q$. The first assertion follows from Corollary~\ref{cor:rankonesing}. 

We can choose the isomorphism $\phi \colon \P^1 \xlongrightarrow{\cong} \bar{Y} \subset \P^q$ to be given by the matrix 
\[
B=\begin{bmatrix}
1&0\\
0&1\\
b_2&c_2\\
\vdots&\vdots\\
b_q&c_q
\end{bmatrix}.
\]
The corresponding arrangement \B\ consists of $q+1$ distinct points in $\P^1$. The Euler characteristic of the complement of \B\ is $1-q$. Then for generic $a$ the \B-master function $\Psi_a$ has $q-1$ isolated, nondegenerate critical points in $\bar{Y}\cap(\C^*)^q$. In fact, since $Y$ is dense in $\bar{Y}\cap(\C^*)^q$, we apply Lemma~\ref{lem:isolated} to see that, for generic $a$, $\Psi_a$ has $q-1$ critical points in $Y$. Then Corollary~\ref{cor:fibers} implies $\crit(\Phi)$ is the union of $q-1$ fibers of $\Phi_\Lambda$.

The projection $\P^q \rightarrowtail \P^1$ along $z_0=z_1=0$ restricts to an isomorphism on $\bar{Y}$. Then the last statement of Theorem~\ref{thm:connected} implies the fibers of $\Phi_\Lambda$ are connected. These fibers are given by $[\Phi_{\nu_0}:\Phi_{\nu_1}]=[a_0:a_1]$. Since the $\Phi_{\nu_i}$ have degree $d$, $\crit(\Phi_\Lambda)$ is a union of $q-1$ connected hypersurfaces of degree $d$. The generic fiber of $\Phi_\Lambda$ is smooth by Bertini's Theorem \cite[Corollary III.10.9]{Har77}.
\end{proof}

\begin{corollary} For generic $\omega=\d \log(\Phi) \in \R^1(\A)$, the number of connected components of $\crit(\Phi_\la)$ is equal to the dimension of $H^1(\AA,\omega)$.
\end{corollary}

Example~\ref{ex:referee} shows that $\crit(\Phi)$ need not be smooth or irreducible for all $\omega=\d \log(\Phi) \in \R^1(\A)$. 

\begin{example} 
\label{ex:hessian2}
By \cite{PerYuz07,Yuz08}, the maximum number of blocks in a multinet is equal to four. The only known example with four blocks is the multinet on the Hessian arrangement corresponding to the Hesse pencil, Example~\ref{ex:hessian}. The factors of the polynomial master functions $P_0=xyz, P_1, P_2, P_3$ define the blocks of a multinet on \A\ with all multiplicities equal to one. These master functions satisfy two linear syzygies:
\begin{align*}
P_2&=3(1-\zeta)P_0+P_1\\
P_3&=3(1-\zeta^2)P_0+P_1.
\end{align*}

Then the variety $\bar{Y}_\Lambda$ corresponding to the basis 
\[
\Lambda=\{\omega_{100}, \omega_{010}, \omega_{001}\}
\] 
is the line given by $z_2=3(1-\zeta)z_0+z_1,\ z_3=3(1-\zeta^2)z_0+z_1$ in $\P^3$, which meets the coordinate hyperplanes in the four points corresponding to the singular fibers. The corresponding arrangement \B\ consists of four points in general position in $\P^1$, given by the rows of the matrix
\[
\begin{bmatrix}
1&0\\
0&1\\
3(1-\zeta)&1\\
3(1-\zeta^2)&1
\end{bmatrix}.
\]

The complement of \B\ has Euler characteristic $-2$, hence a generic \B-master function has two isolated critical points. Then, for generic $a=(a_1,a_2,a_3)$, the critical locus of the \A-master function $\Phi=\Phi_{a_1a_2a_3}$ has two components and codimension one, as found by direct calculation in Example~\ref{ex:hessian}. This example shows that Theorem 4.8(ii) and Corollary 4.9 may not hold under the weaker hypothesis that $a_i\neq 0$ for all $i$.

\end{example}

Here is a rank-four example, that has appeared in different form in the lectures of A.~Libgober in this volume.

\begin{example}
\label{ex:cube}
Let \A\ be the arrangement with defining polynomial 
\[
Q=(x+y)(x-y)(y+z)(y-z)(z+w)(z-w)(w+x)(w-x),
\] 
with the hyperplanes numbered according to the order of factors in $Q$. Then \A\ is a 2-generic subarrangement of the Coxeter arrangement of type $D_4$. Up to lattice-isotopy, the dual projective point configuration consists of the eight vertices of a cube - see Figure~\ref{fig:cube}.
\begin{figure} 
\includegraphics[width=5cm]{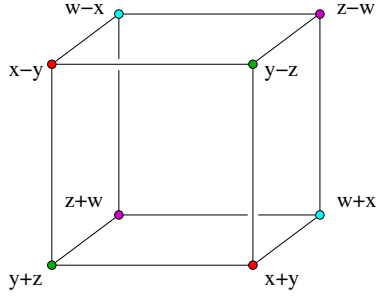}
\caption{A rank-four matroid with a linear syzygy of master functions.}
\label{fig:cube}
\end{figure}
Let $D$ be the subspace of $A^1$ with basis
\[
\Lambda=\{(\omega_0+\omega_1)-(\omega_6+\omega_7), \ (\omega_2+\omega_3)-(\omega_6+\omega_7), (\omega_4+\omega_5)-(\omega_6+\omega_7)\}.
\]
Then $D$ is a rational singular subspace of rank two; if $\omega \in D$, then $H^1(\AA,\omega)=0$ and $H^2(\AA,\omega)\cong D/\C\omega$ has dimension 2. $\bar{Y}_\Lambda$ is the linear hyperplane in $\P^3$ defined by $z_0+z_1+z_2+z_3=0$, reflecting the linear syzygy of polynomial master functions
\[
(x+y)(x-y)+(y+z)(y-z)+(z+w)(z-w)+(w+x)(w-x)=0.
\]
From Proposition~\ref{prop:linhyp} we see that 
\[
\Phi=\left(\frac{x^2-y^2}{w^2-x^2}\right)^{a_1}\left(\frac{y^2-z^2}{w^2-x^2}\right)^{a_2}\left(\frac{z^2-w^2}{w^2-x^2}\right)^{a_3}
\]
has nonempty critical set of codimension two in $\P^3$, for generic $(a_1,a_2,a_3)$.

\end{example}

%
\end{section}

\begin{section}{The rank condition}
\label{sec:trop}
We are left with the problem of finding rational singular subspaces of $A^1$. The theory of multinets gives a method to find such subspaces of rank one. In this section we give a combinatorial condition for a set $\Lambda$ of linearly independent integral one-forms to span a singular subspace of $A^1$ of arbitrary rank, using tropical implicitization and nested sets. 

\subsection{Tropicalization} The {\em tropicalization} of a projective variety $V$ in $\P^q$ is a polyhedral fan $\trop(V)$ in tropical projective space $\TP^q=\bbR^{q+1}/\bbR(1,\dots,1)$, associated to a homogeneous defining ideal $I$ of $V$. If $V$ is a hypersurface with defining equation $f=0$, then $\trop(V)$ is the image  in $\TP^q$ of the union of  
the cones of codimension at least one in the normal fan of the Newton polytope of $f$. In general, $\trop(V)$ is the image of the union of the  
cones of codimension at least one in the Gr\"obner fan of $I$. 
The set $\trop(V)$ arises geometrically from the lowest-degree terms in Puiseux expansions of curves lying in $V$.
See \cite{DFS07} and the references therein for background on tropical varieties. See \cite{Ox92} for matroid terminology.

We will need several results from tropical geometry. The first is a theorem of Bieri and Groves \cite{BG84}.

\begin{theorem}  The maximal cones in $\trop(V)$ have dimension equal to $\dim(V)$. 
\label{thm:dim}
\end{theorem}

If $V$ is an $\ell$-dimensional linear subvariety of $\P^n$, given as the column space of a matrix $R$,  then the tropicalization $\trop(V)$ depends only on the dependence matroid $\G$ on the rows of $R$. In our setting the rows of $R$ give the defining forms of a hyperplane arrangement \A. The {\em matroid polytope} $\Delta(\G)$ of $\G$ is the convex hull of the set
\[
\big\{\sum_{i\in B} e_i \mid B \ \text{is a basis of} \ \G\big\}.
\]
The tropicalization $\trop(V)$, called the {\em Bergman fan} of \G, is the image in $\TP^n$ of the union of the 
cones of codimension 
at least one in the normal fan of $\Delta(\G)$.  We denote it by $B(\G)$. 

In \cite{FeStu05} the Bergman fan is described in terms of {\em nested set cones}. Let $\mathcal G$ be the set of proper connected (i.e., irreducible) flats of \G. These are the flats corresponding to the dense edges of the projective arrangement \A. A collection $\S=\{X_1,\ldots,X_p\}$ of subsets of $\mathcal G$ is a {\em nested set} if, for every set \T\ of pairwise incomparable elements of \S, the join $\bigvee \T$ is not an element of $\mathcal G$. The nested sets form a simplicial complex $\Delta=\Delta(\G)$, the {\em nested set complex}, which is pure of dimension $r=\ell-1$. It is the coarsest of a family of nested set complexes, obtained by replacing $\mathcal G$ with larger ``building sets." All of these complexes are subdivided by the order complex of the poset of nonempty flats of \G. 

If $S\subseteq \A$, set $e_S=\sum_{H_i\in S} e_i$. The {\em nested set fan} $N(\G)$ is the image in $\TP^n$ of the union of the cones generated by 
\[
\{e_S \mid S \in \S\}
\]
for $\S \in \Delta(\G)$.
From \cite{FeStu05} we have the following result.
\begin{theorem} The nested set fan $N(\G)$ subdivides the Bergman fan $\B(\G)$.
\label{thm:nested}
\end{theorem}

\subsection{Singular subspaces} Let \A\ be an arrangement in $\P^\ell$ with homogeneous defining linear forms $\alpha_0,\ldots, \alpha_n$. Let \G\ be the underlying matroid of \A, the dependence matroid on $\{\alpha_0, \ldots, \alpha_n\}$. Suppose $D$ is a rational subspace of $A^1(\A)$, with integral basis $\Lambda=\{\omega_{\xi_1},\ldots, \omega_{\xi_q}\}$. We identify $\Lambda$ with the $q \times (n+1)$ matrix of integers $\begin{bmatrix}\xi_{ij}\end{bmatrix}$, and recall that $\sum_{j=0}^n \xi_{ij}=0$ for $1\leq i\leq q$.
 Let $\bar{Y}_\Lambda$ be the Zariski closure of the image of the associated rational map $\Phi_\Lambda=[1:\Phi_{\xi_1}:\cdots : \Phi_{\xi_q}] \colon \P^\ell \rightarrowtail \P^q$.

The main observation is that $\Phi_\Lambda$ can be factored as a linear map followed by a monomial map. Assume \A\ is essential, and let 
\[
\alpha =\begin{bmatrix}\alpha_0 : \cdots : \alpha_n\end{bmatrix}\colon \P^\ell \to \P^n.
\]
Let $\mu=\mu_\Lambda \colon \P^n \rightarrowtail \P^q$ be given by 
\[
\mu([t_0:\cdots:t_n])=[1:t^{\xi_1}:\cdots:t^{\xi_q}],
\]
where we use the usual vector notation for monomials: $t^{(i_0,\ldots,i_n)}=t_0^{i_0}\cdots t_n^{i_n}$. Then the following diagram commutes.
\[
\xymatrix{
&{\P^n \vspace*{3pt}}\ar@{>->}[dr]^{\mu_\Lambda}&\\
\P^\ell \  \ar@{>->}[rr]^{\Phi_\Lambda}\ar[ur]^\alpha&\ar[r]&\P^q
}
\]

In this situation the diagram tropicalizes faithfully, in the following sense. 
\begin{theorem}[{\rm \cite[Theorem 3.1]{DFS07}}] The tropicalization $\trop(\bar{Y}_\Lambda)$ is equal to the image of the Bergman fan $B(\G)$ under the linear map 
\[
\TP^n \to \TP^q
\]
with matrix $\Lambda$.
\label{thm:trop}
\end{theorem}

We obtain the following characterization. Write $\Lambda=\begin{bmatrix} \Lambda_0  | &\cdots& |  \Lambda_n
\end{bmatrix}$ with $\Lambda_j\in \Z^q$ for each $j$. 
For $S=\{H_{j_1},\ldots, H_{j_k}\} \subseteq \A$, let $\Lambda_S =\sum_{r=1}^k \Lambda_{j_r}$.

\begin{theorem} The subspace $D$ is singular if and only if the rank of the matrix
\[
\Lambda_\S=\begin{bmatrix} \Lambda_{S_1}  | &\cdots& |  \Lambda_{S_{\ell-1}}
\end{bmatrix}
\]
is less than $q$, for each maximal nested set $\S \in N(\G)$. In this case the rank of $D$ is the maximal rank of $\Lambda_\S$ for $\S\in N(\G)$.
\label{thm:rank}
\end{theorem}

\begin{proof}  The subspace $D$ is singular if and only if $\dim \bar{Y}_\Lambda<q$. By Theorem~\ref{thm:dim}, this occurs if and only if $\dim \trop(\bar{Y}_\Lambda)<q$. The cones of $\trop(\bar{Y}_\Lambda)$ are images of the cones of $B(\G)$ under $\Lambda$, by Theorem~\ref{thm:trop}. The linear hulls of the cones in $B(\G)$ are the images in $\TP^n$ of the linear spans of  the sets $\{e_S \mid S \in \S\}$, for $\S \in N(\G)$, by Theorem~\ref{thm:nested}. Since $\Lambda(e_S)=\Lambda_S$, the result follows. The last statement holds because the rank of $D$ is equal to $\dim \bar{Y}_\Lambda$.
\end{proof}

\begin{example}
\label{ex:prism}
Consider the arrangement of rank four with defining polynomial
\[ Q=xyz(x+y+z)w(x+y+w), \]
with hyperplanes ordered according to the given factorization of $Q$. The dual point configuration consists of the six vertices of a triangular prism in $\P^3$. 

For generic $(a,b,c)$, the master function
$\Phi=x^ay^{-a}z^b(x+y+z)^{-b}w^c(x+y+w)^{-c}$ has critical set of codimension two, and $H^2(A,\omega)\cong \C$ for $\omega=\d \log(\Phi)$. Based on our other examples, one might suspect that the subspace $D$ with basis $\Lambda=\{\omega_0-\omega_1, \ \omega_2-\omega_3, \ \omega_4-\omega_5\}$ is singular. 
Among the nested sets of \A\ is the set $\S=\{0,02,024\}$, and 
\[
\Lambda_\S=\begin{bmatrix}1&1&1\\0&1&1\\0&0&1
\end{bmatrix}
\]
does not have rank two. Then by Theorem~\ref{thm:rank}, $D$ is not singular. In fact, $\psi= a\cdot\partial(\omega_0\omega_1\omega_5)+b\cdot\partial(\omega_2\omega_3\omega_5)$ is the unique 2-cocycle for $\omega$. $\psi$ is trivial if $a$ or $b$ is zero, and our argument shows that $\psi$ is not decomposable if $a$ and $b$ are both nonzero.

\end{example}

In the forthcoming paper \cite{BFW11}, Theorem~\ref{thm:rank} is used to derive combinatorial conditions for $p$-generic arrangements to support singular subspaces of rank $p$. Using that approach one can show by combinatorial means that there are no singular subspaces of $A^1$ of rank two in Example~\ref{ex:prism}.

Theorem~\ref{thm:rank} also has the following corollary.

\begin{corollary} If $\G_1$ and $\G_2$ are loop-free matroids on the ground set $\{1,\ldots, n\}$ and $B(\G_1)=B(\G_2)$, then $\G_1=\G_2$.
\end{corollary}

\begin{proof} Let \G\ be a loop-free matroid on $\{1,\ldots, n\}$, with Orlik-Solomon algebra $\AA=\AA(\G)$. Let $e_1,\ldots, e_n \in A^1$ denote the canonical generators. Then $S=\{i_1,\ldots,i_k\}\subseteq \{1,\ldots,n \}$ is dependent in \G\ if and only if $e_{i_1}\we \cdots \we e_{i_k}=0$ in $A^k$. (This statement holds even if \G\ has multiple points.) Equivalently, $S$ is dependent if and only if the coordinate subspace $D\subseteq A^1$ spanned by $\{e_{i_1},\ldots,e_{i_k}\}$ is singular. By Theorem~\ref{thm:trop}, $D$ is singular if and only if the image of the Bergman fan $B(\G)\subseteq \TP^n$ under the projection $\TP^n \to \TP^S \cong \TP^{k-1}$ has dimension less than $k-1$. Thus $B(\G)$ determines \G.
\end{proof}

\end{section}
\begin{ack} This work was begun during the semester-long program on Hyperplane Arrangements at MSRI in Fall, 2004. The authors are grateful to the institute for support. We thank Bernd Sturmfels for pointing out the relevance of Theorem~\ref{thm:trop}. We also thank an anonymous referee for pointing out the phenomena illustrated in Example~\ref{ex:referee}, leading us to correct an earlier version of Theorem 3.5.
\end{ack}

\newcommand{\arxiv}[1]
{\texttt{\href{http://arxiv.org/abs/#1}{arxiv:#1}}}

\renewcommand{\MR}[1]
{\href{http://www.ams.org/mathscinet-getitem?mr=#1}{MR#1}}

\newcommand{\web}[1]{{\texttt{\href{http://#1}{#1}}}}

\bibliographystyle{plain}

\end{document}